\documentclass[hidelinks,onefignum,onetabnum]{siamart250211}
\usepackage{enumitem}

\newcommand{\R}{\mathbb R}

\newcommand{\Et}{\widetilde{\mathbb E}}

\newcommand{\osc}{\operatorname{osc}}

\newcommand{\Lip}{\operatorname{Lip}}
\newcommand{\cE}{\mathcal E}

\newcommand{\cL}{\mathcal L}
\newcommand{\cP}{\mathcal P}

\newcommand{\dist}{d_{S^1}}
\newcommand{\PP}{\mathbb P}
\newcommand{\FF}{\mathcal F}

\usepackage{lipsum}
\usepackage{amsfonts}
\usepackage{graphicx}
\usepackage{epstopdf}
\usepackage{algorithmic}

\usepackage{amsmath, amssymb, mathrsfs, mathtools}
\usepackage{dsfont}
\ifpdf
  \DeclareGraphicsExtensions{.eps,.pdf,.png,.jpg}
\else
  \DeclareGraphicsExtensions{.eps}
\fi

\newcommand{\normmm}[1]{{\left\vert\kern-0.25ex\left\vert\kern-0.25ex\left\vert #1
   \right\vert\kern-0.25ex\right\vert\kern-0.25ex\right\vert}}

\newsiamremark{remark}{Remark}
\newsiamremark{hypothesis}{Hypothesis}
\crefname{hypothesis}{Hypothesis}{Hypotheses}
\newsiamthm{claim}{Claim}
\newsiamremark{fact}{Fact}
\crefname{fact}{Fact}{Facts}
\newsiamthm{example}{Example}
\newtheorem{assumption}{Assumption}

\headers{The existence of ISE for $G\text{-}$SDEs}{D.Y. Zhao and H.Z. Zhao}

\title{The existence of invariant sublinear expectations for $G\text{-}$SDEs}

\author{Danyang Zhao\thanks{ School of Mathematics, Shandong University, Jinan, 250100, China
  (\email{dyang.z@mail.sdu.edu.cn}
  ).}
\and Huaizhong Zhao\thanks{Department of Mathematical Sciences, Durham University, Durham, DH1 3LE, U.K. and Research Center for Mathematics and Interdisciplinary Sciences, Shandong University, Qingdao, 266237, China
  (\email{huaizhong.zhao@durham.ac.uk})}
}

\usepackage{amsopn}

\ifpdf
\hypersetup{pdftitle={The existence of invariant sublinear expectations for $G$-SDEs}, pdfauthor={Danyang Zhao and Huaizhong Zhao}}
\fi

\begin{document}

\maketitle

\begin{abstract}
In this paper, we study the existence of invariant sublinear expectations of Markovian semigroups on sublinear expectation spaces. 
To achieve this, we establish a complete metric space of sublinear expectations, on which we extend Harris' method to the nonlinear setting on the convergence of sublinear semigroups. We then explore two cases of $G\text{-}$diffusions by studying the Lyapunov function and the local Doeblin condition. One is the $G\text{-}$Brownian motion on the unit circle which is the case studied in Feng and Zhao \cite{Zhaonon}, but with the new method. Another is the multidimensional $G\text{-}$SDEs on the whole space $\mathbb{R}^d$. We establish, for the first time in the literature, the existence of the invariant sublinear expectation for $G\text{-}$SDEs under the non-degenerate and weakly dissipative assumption. For this, we prove that for a class of $G\text{-}$SDEs, the $G\text{-}$expectation can be represented as the supremum of the semigroup of a family of SDEs, of which the regularity is obtained by considering  { oscillations and coupling of the nonlinear Markovian semigroup,} and the Denis-Hu-Peng representation for the distribution of $G\text{-}$Brownian motions.
\end{abstract}



\section{Introduction}
It is well known that the existence of invariant measures is central in the theory of dynamical systems and essential to the study of ergodic theory (Birkhoff \cite{brikhoff}, von Neuman \cite{vonneumann}, Walters \cite{inergodic}, Einsiedler and Ward \cite{ergodicnumber}). These measures describe the long-term behaviour of dynamical systems and their statistical equilibrium and are widely used in the study of physics, climate dynamics, ecology, biology, economics, etc.

 Since the work of Kryloff and Bogolyuboff in 1937 (\cite{1937}), numerous researchers have made significant contributions to the study of the existence of invariant measures for stochastic systems, e.g., Da Prato and Zabczyk \cite{Daprato}, Durrett \cite{durrett}, Meyn and Tweedie \cite{markovstability}, Hairer and Mattingly \cite{Hairer2006}, to name but a few. However, there are still few relevant studies on the dynamics and ergodicity of stochastic systems with probability model uncertainties. Such uncertainty is inevitably present in many real world problems in which probability is very often undetermined, for example, higher-level uncertainty known as Knightian uncertainty in economics (\cite{knight}). In statistics literature, probability model uncertainties are modelled by upper expectations or upper provisions (Huber \cite{Huber1981}, Walley \cite{walley1991}). Dynamic models such as Markov chains (Cooman et al. \cite{imc2009}) and $G\text{-}$diffusions (Peng \cite{pengbook}) have been investigated. But what is the long-time dynamical behaviour and what equilibrium would such a system settle eventually? To address such a critical question in the sense of distribution, the study of invariant distributions, called invariant sublinear expectations (ISE), under uncertainty of probability models is a critical research problem.   
 
Probability uncertainties have been studied intensively in the last few decades. In 1989, Schmeidler \cite{1989SubjectivePA} relaxed the classical assumption of additive probabilities and introduced the Choquet integral to model the decision-making problem under uncertainty. 
 In 1997, based on the study of BSDEs, Peng \cite{peng1997Backward} introduced the concept of $g\text{-}$expectation, a nonlinear expectation that generalises the classical expectations by incorporating uncertainty and dynamic risk assessment. This concept was used in modelling the dynamic uncertainty and risk (Chen and Kulperger \cite{chenprice}, Chen and Epstein \cite{chen}). In 1999, Artzner et al. \cite{coherent1999} introduced four mathematical axioms that a coherent risk measure must satisfy, one of which is sub-additivity. Their method laid the foundation for financial mathematics under uncertainty. Follmer and Schied (\cite{FandS2004}, \cite{FandS2002}) generalised coherent risk measure to convex risk measures, establishing that sublinearity arises naturally in financial markets and underpins pricing, risk measurement, and hedging under uncertainty. It was proven that a coherent risk measure is equivalent to a sublinear expectation (SE). In 2004, Peng \cite{Peng2004} introduced the $G\text{-}$framework, which has desirable properties and topology, and has become an effective tool for analysing the uncertainty of dynamic problems (Peng \cite{Filpeng,Peng2005,peng2017,pengbook}).
 

Our study is motivated by existing research on the ergodicity of ISEs, which is defined as, for any invariant set, it has capacity $0$ or its complement has capacity $0$ (Feng and Zhao \cite{Zhaonon}). Feng et al. \cite{Zhaospa} is the discrete analogue of \cite{Zhaonon}. Other related results include De Cooman et al. \cite{imc2009} on a sufficient condition under which the limit upper expectation of imprecise Markovian chains with finite state space is the unique invariant upper expectation; T’Joens and de Bock \cite{imc2021} on the concept of weak ergodicity and a necessary and sufficient condition for weak ergodicity using graphical models in a finite state space. Cerreia-Vioglio et al. \cite{ergodiclower} studied a regime that for any invariant set, it has capacity either $0$ or $1$ for continuous lower probabilities and provided Birkoff's law of large numbers, while Sheng and Song \cite{ergodicuppersong} provided some further studies. The $\{0,1\}$-regime does not lead to irreducibility, but is equivalent to finite ergodic component decomposition (Feng et al. \cite{fengliuhuangzhao}).
In the continuous-time case, Feng and Zhao \cite{Zhaonon}
established the theoretical
framework for the ergodic theory of SE dynamical systems without imposing any continuity condition and proved the equivalence of irreducibility and their $\{0,0\}$\text{-}regime. This framework is applicable to G-diffusion processes. There are only two results for the existence of ISE for G-diffusions. One is the $G\text{-}$Brownian motion on the unit circle (\cite{Zhaonon}). The other is a $G\text{-}$diffusion process under a strong dissipative assumption (Hu et al. \cite{Huergodic}). The problem of the existence of ISE, in general, remains unsolved.

In this paper, we first study the existence of ISE for semigroups of sublinear Markovian chains. We assume a weak contraction condition for a Lyapunov function, i.e., contraction outside of a bounded set and a Doeblin-like condition on a bounded set. Although in the classical Brownian case, the existence and ergodicity of invariant measures for a wide range of diffusions have been established, it has not been possible to establish similar results to the case with $G\text{-}$Brownian motions. There are many fundamental obstacles due to the lack of both probabilistic and analytic tools. Ergodic theory under nonlinear expectations is still in its infancy, and a general result on the convergence of sublinear expectation operators is very much needed. Although the connection of fully nonlinear PDEs and $G\text{-}$Markovian processes is established in Peng \cite{pengbook}, many important properties that are readily available in the linear case remain elusive for G-Brownian motions when studying ergodicity. The present work makes substantial advances in bridging these theoretical gaps, with results applicable to a broad class of sublinear Markovian semigroups (SMSs).

     In contrast to previous research, we do not need to examine the existence of limits of expected values in a sublinear setting. Instead, we establish a complete metric space of sublinear expectation operators in which we provide a sufficient condition for the existence of the limit of SMSs. Moreover, the limit is a unique invariant expectation for the related SMS. Our proof relies on extending the work of Hairer and Mattingly \cite{Yet} and Harris' small set recurrence idea to the sublinear setting. We then study two specific $G\text{-}$diffusions. The first involves the $G\text{-}$Brownian motion on the unit circle. We prove the existence of its ISE using our method under the non-degenerate assumption. It is noted that Feng and Zhao \cite{Zhaonon} proved the same result; however, their work, while pioneering, involves explicit computation, making it difficult to apply to general cases. The second case is to consider a multidimensional $G\text{-}$SDE for which we assume a weakly dissipative condition on the drift. In this part of the proof, we show that the associated invariant expectation exists uniquely.
    
     In order to deal with general $G\text{-}$diffusions, some novel probabilistic and analytic tools are developed. We extend the results of Section 3 in Denis et al. \cite{DenishU} in two aspects. First, we show that the distribution of a class of G\text{-}SDEs, whose $G\text{-}$Brownian component is additive and non-degenerate, can be represented as the supremum of the linear expectations of a sequence of SDEs. Second, we relax the requirement that the relevant test function must be Lipschitz continuous. Instead, we show that as long as the test function is bounded and measurable, the G\text{-}expectation can still be represented as mentioned above. A key step to make this possible is to get a regularity result { for the nonlinear semigroup using oscillations and coupling of the nonlinear Markovian semigroup} for all test functions due to the uniform non-degeneracy of all possible choice of processes for the diffusion coefficients. This observation leads to the continuity of the sublinear semigroup in both space and time variables.

The rest of this paper is organised as follows. In Section 2, we recall some basic definitions that will be used later. In Section 3, we prove the existence of ISE in SE spaces, considering both discrete-time and continuous-time cases. In Section 4, we present examples of $G\text{-}$Brownian motions on the unit circle and multidimensional $G\text{-}$SDEs. Finally, some more proofs are provided in Section 5.

\section{Preliminaries}
In this section, we recall some basic definitions in SE spaces that are used later.  

Let $\Omega$ be a given set and $\mathcal{H}$ be a linear space of real-valued functions defined on $\Omega$, where $\mathcal{H}$ satisfies the following two conditions,\\
(1). $c\in \mathcal{H}$ for each constant $c$;\\
(2). $\vert X \vert \in \mathcal{H}$, if $X \in \mathcal{H}$.    \\
Thus, $\mathcal{H}$ can be taken as a space of random variables.

\begin{definition}[Sublinear expectation \cite{pengbook}]
 A sublinear expectation $\mathbb{E}$ is a functional $\mathbb{E}$: $\mathcal{H} \mapsto \mathbb{R}$ satisfying \\
 \rm{(1).} Monotonicity. $\mathbb{E}[X] \leq \mathbb{E}[Y], \ \mathrm{if}\  X \leq Y$.\\
 \rm{(2)}. Constant preserving. $ \mathbb{E}[c] = c\ \mathrm{for}\ c\in \mathbb{R}$.\\
\rm{(3)}. Sub-additivity. For each $X,Y \in \mathcal{H}$, $ \mathbb{E}[X+Y] \leq  \mathbb{E}[X]+ \mathbb{E}[Y]$.\\
 \rm{(4)}. Positive homogeneity. $\mathbb{E}[\lambda X] =  \lambda\mathbb{E}[X]\ \mathrm{for}\ \lambda \geq 0$.\\
\end{definition}

 The triplet $(\Omega, \mathcal{H}, \mathbb{E})$ is called a sublinear expectation space. When studying SE spaces, one often considers the following SE space $(\Omega, \mathcal{H},\mathbb{E})$. It satisfies the requirement that if $X_1,\dots,X_d \in \mathcal{H}$, then $\varphi(X_1,\dots,X_d)\in \mathcal{H}$ for any $\varphi \in C_{l.lip}(\mathbb{R}^d)$, where $C_{l.lip}(\mathbb{R}^d)$ denotes the linear space of functions $\varphi$ satisfying a local Lipschitz condition. The d-dimensional random variable $X=(X_1,\dots,X_d)$ is called a random vector, denoted by $X\in \mathcal{H}^d$. $C_{l.lip}(\mathbb{R}^d)$ can be replaced by any one of the following spaces of functions defined on $\mathbb{R}^d$ (\cite{pengbook}),\\
 $\bullet \  L^0(\mathbb{R}^d)$: the space of Borel measurable real-valued functions;\\
  $\bullet \  B_b(\mathbb{R}^d)$: the space of bounded functions in $L^0(\mathbb{R}^d)$;\\
   $\bullet \  C_b(\mathbb{R}^d)$: the space of continuous functions on $B_b(\mathbb{R}^d)$;\\
 $\bullet \  C_{b,lip}(\mathbb{R}^d)$: the space of real-valued bounded and Lipschitz continuous functions;\\
 $\bullet \  C_b^k(\mathbb{R}^d)$: the space of bounded, $k\text{-}$times continuously differentiable functions with bounded derivative of all orders less than or equal to $k$. \ \\

 Let $(\Omega, \mathcal{F})$ be a measurable space and $\mathcal{D}= L_b(\mathcal{F})$, the space of all $\mathcal{F}\text{-}$measurable real-valued functions such that ${{\rm sup}}_{\omega \in \Omega}\vert X(\omega)\vert < \infty$. It is easy to see (1). $1\in \mathcal{D}$; (2). $\vert X \vert \in \mathcal{D}$, if $X\in \mathcal{D}$. Then the triplet $(\Omega, \mathcal{D},\mathbb{E})$ is an SE space. Similarly, if $X_1,\dots,X_d \in \mathcal{D}$, then $\varphi(X_1,\dots,X_d)\in \mathcal{D}$ for any $\varphi \in B_b(\mathbb{R}^d)$. For any d-dimensional random vector $X \in \mathcal{D}^d$, the sublinear distribution of $X$ under $\mathbb{E}[\cdot]$ is defined by
 \begin{equation}
     T[\varphi]\coloneqq \mathbb{E}[\varphi(X)],\ \varphi \in B_b(\mathbb{R}^d).
 \end{equation}
 This distribution $T[\cdot]$ is again an SE defined on $B_b(\mathbb{R}^d)$.

 Consider a family of SE operators parameterized by $t\in \mathbb{R^{+}}$,
 \begin{equation*}
     T_t:\  B_b(\mathbb{R}^d)\rightarrow B_b(\mathbb{R}^d),\ t\geq 0.
 \end{equation*}
 
\begin{definition}[Sublinear Markovian semigroup \cite{Peng2005}]\label{deftt}The operator $T_t$ is called a sublinear Markovian semigroup if for any $\varphi \in B_b(\mathbb{R}^d)$, it satisfies\\
\rm{(1)}. For each fixed $(t,x)\in \mathbb{R}^+ \times \mathbb{R}^d$, $(T_t\varphi)(x)$ is an SE defined on $B_b(\mathbb{R}^d)$;\\
\rm{(2)}. $T_0\varphi=\varphi$;\\
\rm{(3)}. $t \to T_t\varphi$ satisfies the following Chapman semigroup formula, $(T_t \circ T_s)(\varphi)=T_{t+s}\varphi,\ t,s\geq 0.$
\end{definition}
 
We will present here an example of an SMS, which is closely related to the content in Section 4. In this example, we take $\varphi \in C_{b.lip}(\mathbb{R}^d)$, and later we will prove that $\varphi$ can be taken from $B_b(\mathbb{R}^d)$.
 
\begin{example}(Peng \cite{pengbook})\label{Gexample}
Define $\Omega=C^d_0(\mathbb{R}^+)$: the space of all $\mathbb{R}^d\text{-}$valued continuous paths $(\omega_t)_{t\in \mathbb{R}^+}$, with $\omega_0=0$, equipped with the distance
 \begin{equation*}
     \rho(\omega^1,\omega^2)\coloneqq \sum^{\infty}_{i=1}2^{-i}[(\underset{t \in[0,i]}{\rm max}\vert \omega^1_t-\omega^2_t\vert)\wedge1].
 \end{equation*}
 Then $(\Omega,\rho)$ is a complete separable metric space.

 Let $S(d)$ be the collection of symmetric $d\times d$ matrices and $S_{+}(d)$ the collection of positive definite symmetric $d\times d$ matrices. Assume $G:S(d)\to R$ is a given sublinear function which is monotonic on $S(d)$. Then there exists a bounded, convex and closed subset $\Sigma \subset S_{+}(d)$ such that 
\begin{equation*}
    G(A)=\underset{B\in \Sigma}{{\rm sup}}[\frac{1}{2}tr(AB)],\  for\  A\in S(d).
\end{equation*}
Let 
\begin{equation*}
     Lip(\Omega)\coloneqq\{ \varphi(\omega_{t_1}, \omega_{t_2},\cdots, \omega_{t_n}):\forall n\geq 1, \ t_1,\cdots,t_n\in \mathbb{R}^{+},\ \forall \varphi \in C_{b,lip}((\mathbb{R}^{d})^n)\}.
 \end{equation*}
Then the $G\text{-}$normal distribution $N(\{0\}\times \Sigma)$ on $(\Omega,Lip(\Omega))$ exists, i.e. there exists a d-dimensional random vector $X$ on an SE space $(\Omega,\mathcal{D},\mathbb{E})$ satisfying
\begin{equation*}
    aX+b\bar{X}\overset{d}{=}\sqrt{a^2+b^2}X,\ for\ a,b\geq0,
\end{equation*}
where $\bar{X}$ is an independent copy of $X$ and $G(A)=\mathbb{E}[\frac{1}{2}\langle AX,X\rangle]$. In \cite{pengbook}, it has been proved that there exists a weakly compact family of probability measures $\mathcal{P}$ on $(\Omega,\mathcal{B}(\Omega))$ such that 
\begin{equation*}
    \mathbb{E}[X]=\underset{P\in\mathcal{P}}{\rm max}E_P[X],\ for\ X\in Lip(\Omega).
\end{equation*}
 Its canonical path is $G\text{-}$Brownian motion $\{B_t\}_{t\geq 0}$ on the SE space $(\Omega,\mathcal{D},\mathbb{E})$ satisfying
 \rm{(1)}. $B_0(\omega)=0$;
 \rm{(2)}. For each $t,s \geq 0$, the increment $B_{t+s}-B_t$ is $N(\{0\}\times s\Sigma)$ distributed and independent of $(B_{t_1},B_{t_2},\dots,B_{t_n})$, for each $n \in \mathbb{N}$ and $0 \leq t_1 \leq t_2 \leq \dots \leq t_n < t$.

 For each fixed $\varphi \in C_{b,lip}(\mathbb{R}^d)$, the function 
 \begin{equation*}
     u(t,x)\coloneqq \mathbb{E}[\varphi(x+B_t)],\ (t,x) \in [0,\infty)\times \mathbb{R}^d,
 \end{equation*}
 is the viscosity solution of the following $G\text{-}$heat equation
 \begin{equation*}
     \frac{\partial}{\partial t}u=G(D^2u),\ u(0,\cdot)=\varphi(\cdot).
 \end{equation*}
 Then $(T_t\varphi)(x)=u(t,x)$ defines an SMS.
\end{example}

According to Example \ref{Gexample}, the space of finite-dimensional cylinderical random variables that will be used later is introduced, for each $T\geq 0$,
 \begin{equation*}
     Lip(\Omega_T)\coloneqq\{ \varphi(B_{t_1}, B_{t_2},\cdots, B_{t_n}):\forall n\geq 1,\  t_1,\cdots,t_n\in [0,T],\  \varphi \in C_{b,lip}(\mathbb{R}^{d\times n})\},
 \end{equation*}
and we rewrite $ Lip(\Omega)\coloneqq \cup_{T}Lip(\Omega_T)$.
 
  Then a consistent SE called $G\text{-}$expectation can be constructed on $Lip(\Omega)$,  still denoted by $\mathbb{E}[\cdot]$, and $B_1$ is $G\text{-}$normally distributed under $\mathbb{E}[\cdot]$. The topological completion of $Lip(\Omega_T)$ (resp. $Lip(\Omega)$) under the Banach norm $\mathbb{E}[\vert \cdot \vert]$ is denoted by $L^1_{G}(\Omega_T)$ (resp. $L_G^1(\Omega)$). Thus, $\mathbb{E}[\cdot]$ is extended uniquely to an SE on $L^1_{G}(\Omega)$ (\cite{pengbook}). The space $L^p_G$ is the completion of $Lip(\Omega)$ under the norm $\Vert \xi \Vert_{L^p_G}\coloneqq (\mathbb{E}[\vert\xi\vert^p])^{\frac{1}{p}}$, for $p \geq 1$.

 In the next section, we will also consider the SMS parameterized by $k \in \mathbb{N}_{+}$, i.e.,
\begin{equation}\label{distdifinetion}
   \bar{T}^k: B_b(\mathbb{R}^d) \to B_b(\mathbb{R}^d), k\in \mathbb{N}.
\end{equation}
And the operator $\bar{T}^k$ satisfies the property in Definition \ref{deftt}, i.e., 
\rm{(1)}. For each fixed $(k,x)\in \mathbb{N}_{+}\times \mathbb{R}^d$, $(\bar{T}^k\varphi)(x)$ is an SE defined on $B_b(\mathbb{R}^d)$;
\rm{(2)}. $\bar{T}^0\varphi=\varphi$;
\rm{(3)}. $k\to \bar{T}^k\varphi$ satisfies $(\bar{T}^k \circ \bar{T}^l)(\varphi)=\bar{T}^{k+l}\varphi$ for $k,l \in\mathbb{N}_{+}$.
We can also give an example.
\begin{example}
    Let $(\Omega,\mathcal{F})$ be a measurable space, and $\mathcal{D}=L_b(\mathcal{F})$.  $\mathbb{E}[\cdot]$ is an SE and then $(\Omega, \mathcal{D}, \mathbb{E})$ is an SE space. Let $Z_k \in \mathcal{D}^d$ be a d-dimensional random vector in time $k \in \mathbb{N}_{+}$ with $Z_0=0$. Moreover, we assume that $Z_k$ has independent increments with additive property $Z_{k+l}=Z_k+Z_l$, for any $k,l \in\mathbb{N}_{+}$. Then for any $\varphi \in B_b(\mathbb{R}^d)$, $x \in \mathbb{R}^d$, we set $ (\bar{T}^k\varphi)(x)=\mathbb{E}[\varphi(x+Z_k)]$.
    It is easy to prove that $\bar{T}^k$ is an SMS parameterized by $k \in \mathbb{N}_{+}$.
\end{example}

\begin{definition}[Invariant sublinear expectation \cite{Zhaonon}, \cite{Huergodic}]
 An invariant sublinear expectation $\tilde{T}:\  B_b(\mathbb{R}^d)\ \rightarrow \  \mathbb{R}$ is an SE satisfying 
 \begin{equation*}
     (\tilde{T}T_s)(\varphi)=\tilde{T}(\varphi),\ for\ any\ \varphi\in  B_b(\mathbb{R}^d),
 \end{equation*}
 where $T_s$, $s\geq 0$ is an SMS.
\end{definition}

\section{Existence of ISEs---a general theorem}
In this section, we discuss sufficient conditions for the existence of ISE in SE spaces. Rather than directly considering the existence of the limit of sublinear expected values, we construct a complete metric space for SE operators in which we investigate whether their limits exist.  We start with the discrete Markovian chain case and then extend it to the continuous case. Notably, the proof for the discrete case can be readily generalized to that of the continuous case. 

In the discrete case, the SE space $(\Omega,\mathcal{D},\mathbb{E})$ and the d-dimensional random vector are defined as before. We consider a family of SEs parameterized by $k \in \mathbb{N}_+$, $\bar{T}^k:B_b(\mathbb{R}^d) \to B_b(\mathbb{R}^d)$ that satisfies all the desired properties in Definition \ref{deftt}.

In 2011, Hairer and Mattingly \cite{Yet} presented a proof of a refinement of Harris' ergodic theorem of Markovian chains. {It can be proved that if $\bar{T}$ satisfies the following assumptions.

(i).
  There exists a function $V : \mathbb{R}^d \to [0,\infty)$ and constants $K \geq 0$, $\gamma \in (0,1)$ such that for all $x \in \mathbb{R}^d$
  \begin{equation}
    (\bar{T}V)(x)\leq \gamma V(x)+K.
  \end{equation}
(ii).  There exists a constant $\alpha \in (0,1)$, a probability measure $\nu$ and a constant $n\in \mathbb{N}$, such that for all $k \geq n$, we have
  \begin{equation}
    \underset{x \in C}{{\rm inf}}-(\bar{T}^k(- \varphi))(x)\geq \alpha E_{\nu}\varphi,
  \end{equation}
  where $\varphi \in B_b(\mathbb{R}^d)$ and $\varphi \geq 0 $, $C = \{x \in \mathbb{R}^d: V(x)\leq R\}$ for some $R > 2K/(1-\gamma)$.
  
The nonlinear version of the Harris small set theorem can also hold an invariant sublinear expectation exists uniquely. However, it is not easy to verify the local Doeblin condition in the sublinear setting (Assumption (ii)). In the following, we develop a coupling approach.} 

{ 
Suppose that a nonlinear transition operator $T$ has the upper
representation
\begin{equation}\label{eq:upper-representation}
 T\varphi(x)
 =\sup_{\mu\in\cP(x)}\mu(\varphi),
 \qquad
 \mu(\varphi):=\int_{\Omega}\varphi(z)\,\mu(dz),
\end{equation}
where $\cP(x)$ is a nonempty family of probability measures. We assume
throughout that all displayed expectations are finite.  Formula
\eqref{eq:upper-representation} makes $T$ monotone, constant preserving,
positively homogeneous and subadditive.

The word ``admissible'' in the next definition is essential.  For a
sublinear operator, it is not sufficient to couple one conveniently chosen
pair of measures: the first marginal must be allowed to range through every
law contributing to the supremum in \eqref{eq:upper-representation}.

\begin{definition}\label{def:directed-coupling}
Fix $x,y\in {\Omega}$ and $\mu\in\cP(x)$.  A directed admissible coupling from
$(x,\mu)$ to $y$ is a pair $(\nu,\pi)$ such that
\[
 \nu\in\cP(y),\qquad \pi\in\Gamma(\mu,\nu),
\]
where $\Gamma(\mu,\nu)$ denotes the set of probability measures on
${\Omega \times \Omega}$ having marginals $\mu$ and $\nu$.
\end{definition}
}

 For the sake of analysis, we recall the following weighted supremum norms on $B_b(\mathbb{R}^d)$ depending on a scale parameter $\beta > 0$,
 \begin{equation}\label{norm1}
    \|\varphi\|_\beta = \underset{x}{{\rm sup}}\frac{|\varphi|}{1+\beta V(x)}.
 \end{equation}
 Define a metric $d_\beta$ between points in $\mathbb{R}^d$ by
 \begin{equation}\label{cost}
   d_{\beta}(x,y)=\left\{
   \begin{aligned}
   &0,  &if\ x=y, \\
   &2+\beta V(x) + \beta V(y),   &if\ x \neq y,
   \end{aligned}
   \right.
 \end{equation}
  and this metric induces a Lipschitz seminorm on $B_b(\mathbb{R}^d)$,
 \begin{equation}\label{norm2}
  \normmm{\varphi}_{\beta}=\underset{x \neq y}{{\rm sup}}\frac{|\varphi(x)- \varphi(y)|}{ d_{\beta}(x,y)}.
 \end{equation}
{In \cite{Yet}, it has been proved that the following translation invariance holds:}
\begin{equation}
  \normmm{\varphi}_{\beta} = \underset{a \in \mathbb{R}}{{\rm inf}}\parallel \varphi + a \parallel_{\beta}.
\end{equation}

In order to prove the convergence of a family of SEs, we also define two metrics for any two SE operators $\bar{T}$, $\tilde{T}$. In the following lemma, we will prove that the space of SE operators $\mathcal{M} \coloneqq \{T: T$ is an SE on $B_b(\mathbb{R}^d)$\} is a complete metric space under these two metrics. For this, set for any $\bar{T},\ \tilde{T}\in \mathcal{M}$, 
 \begin{equation}
   \rho_{\beta}(\bar{T},\tilde{T})=\underset{\varphi:\parallel\varphi\parallel_{\beta}\leq 1}{{\rm sup}}\vert \bar{T}\varphi-\tilde{T}\varphi\vert,
 \end{equation}
 and
 \begin{equation}
   d_{\beta}(\bar{T},\tilde{T})=\underset{\varphi:\normmm{\varphi}_{\beta}\leq 1}{{\rm sup}}\vert \bar{T}\varphi-\tilde{T}\varphi\vert. 
 \end{equation}
 
 By the positive homogeneity of SEs, the metric $\rho_{\beta}$ for SE $\bar{T}$, $\tilde{T}$ can also be represented by
  \begin{equation}
   \rho_{\beta}(\bar{T},\tilde{T})=\underset{\varphi:\parallel\varphi\parallel_{\beta}\neq 0}{{\rm sup}}\frac{\vert \bar{T}\varphi-\tilde{T}\varphi\vert}{\parallel\varphi\parallel_{\beta}},
  \end{equation}
and $d_{\beta}$ also has similar expressions.
\begin{lemma}
  The space $\mathcal{M}$ is a metric space under the metric $\rho_{\beta}$ (and $ d_{\beta}$).
\end{lemma}
\begin{proof}
 The main point is to prove that $\rho_{\beta}$ is a metric on $\mathcal{M}$. It is easy to see that for any $\bar{T}$, $\tilde{T} \in \mathcal{M}$, $\rho_{\beta}(\bar{T},\tilde{T})= \rho_{\beta}(\tilde{T},\bar{T})$ and $\rho_{\beta}(\bar{T},\tilde{T}) \geq 0$. Now, assume $\rho_{\beta}(\bar{T},\tilde{T})=0 $. We have for any $\varphi \in B_b(\mathbb{R}^d)$,
  $\vert \bar{T}\varphi - \tilde{T}\varphi\vert =0$ which is to say that $\bar{T}=\tilde{T}$. So we have proved $\rho_{\beta}(\bar{T},\tilde{T})= 0$ if and only if $\bar{T}\varphi = \tilde{T}\varphi$.
   
   Now, we check the triangle inequality.
  Consider $T$, $\bar{T}$, $\tilde{T} \in \mathcal{M}$, then
\begin{equation*}
  \begin{split}
     &\rho_{\beta}(T,\bar{T})+ \rho_{\beta}(\bar{T},\tilde{T})
     =  \underset{\varphi:\parallel\varphi\parallel_{\beta}\leq 1}{{\rm sup}}\vert T\varphi-\bar{T}\varphi\vert+\underset{\varphi:\parallel\varphi\parallel_{\beta}\leq 1}{{\rm sup}}\vert \bar{T}\varphi-\tilde{T}\varphi\vert \\
      \geq  &\underset{\varphi:\parallel\varphi\parallel_{\beta}\leq 1}{{\rm sup}}\vert T\varphi-\bar{T}\varphi+\bar{T}\varphi-\tilde{T}\varphi\vert
     =  \rho_{\beta}(T,\tilde{T}).
  \end{split}
\end{equation*}
Thus, the triangle inequality holds. We show that $\mathcal{M}$ is a metric space under $\rho_{\beta}$.
It is easy to see, $\rho_{\beta}=d_{\beta}$ as $\{\varphi:\parallel\varphi\parallel_{\beta}\}$ and $\{\varphi:\normmm{\varphi}_{\beta}\}$ only differ by an additive constant and $T(\varphi+c)=T\varphi+c$ for any given constant $c$ and $T \in \mathcal{M}$. Thus, $\mathcal{M}$ is also a metric space under $d_{\beta}$. 
\end{proof}
\begin{lemma}
  The space $\mathcal{M}$ is complete under the metric $\rho_{\beta}$ (and $ d_{\beta}$).
\end{lemma}
\begin{proof}
  Assume $T_n \in \mathcal{M}$, $n \in \mathbb{N}$ and there exists $T_{\infty}$ such that $\rho_{\beta}(T_n, T_{\infty}) \rightarrow 0$, as $n \rightarrow \infty$. Now we prove that $T_{\infty}$ is an SE operator.
  
  First, we check the positive homogeneity. Consider a constant $0<c \leq 1$ and $\varphi \in B_b(\mathbb{R}^d)$ with $\parallel \varphi \parallel_{\beta} \leq 1$, then,
  \begin{equation*}
    \begin{split}
       & \vert T_{\infty}(c\varphi)-cT_{\infty}\varphi\vert \\
         \leq &\vert T_{\infty}(c\varphi)-T_n(c\varphi)\vert + \vert c T_n\varphi-cT_{\infty}\varphi\vert \\
        \leq  &  (1+c)\rho_{\beta}(T_n, T_{\infty}) \  \rightarrow 0, \ as \   n\rightarrow \infty.
    \end{split}
  \end{equation*}
  The third line holds because $\parallel c\varphi \parallel_{\beta}= c \parallel\varphi\parallel_{\beta} \leq 1$.
  Thus, $T_{\infty}(c\varphi) = cT_{\infty}\varphi$. For $c>1$, consider $\frac{1}{c}T_{\infty}(c\varphi)=T_{\infty}(\frac{1}{c}c\varphi)=T_{\infty}\varphi$.
  So, $cT_{\infty}\varphi = T_{\infty}(c\varphi)$ for all $c \geq 0$. With this multiplication result, it can easily be extended to any $\varphi$ with $\parallel\varphi\parallel_{\beta} < \infty$. So $T_{\infty}$ satisfies the positive homogeneity.
  
  Second, we check the sub-additivity. For $\varphi_1$, $\varphi_2 \in B_b(\mathbb{R}^d)$ with $\parallel \varphi_1 \parallel_{\beta} \leq 1$ and $\parallel \varphi_2 \parallel_{\beta} \leq 1$, note that
  \begin{equation*}
       T_{\infty}(\varphi_1+\varphi_2)
      \leq  T_{\infty}(\varphi_1+\varphi_2) -T_n(\varphi_1+\varphi_2)+ T_n\varphi_1+T_n\varphi_2.
  \end{equation*}
  Then
  \begin{equation*}
    \begin{split}
        &T_{\infty}(\varphi_1+\varphi_2)- T_{\infty}\varphi_1- T_{\infty}\varphi_2 \\
        \leq &2\vert T_{\infty}(\frac{1}{2}(\varphi_1+\varphi_2))-T_{n}(\frac{1}{2}(\varphi_1+\varphi_2))\vert + \vert T_n\varphi_1- T_{\infty}\varphi_1 \vert + \vert T_n\varphi_2- T_{\infty}\varphi_2\vert  \\
        \leq & 4\rho_{\beta}(T_n, T_{\infty})\ \rightarrow 0, as \   n\rightarrow \infty.
    \end{split}
  \end{equation*}
    The above inequality holds because $\parallel \frac{1}{2}(\varphi_1+\varphi_2)\parallel_{\beta} \leq 1$. It follows that $T_{\infty}(\varphi_1+\varphi_2) \leq T_{\infty}\varphi_1+ T_{\infty}\varphi_2$ if $\parallel \varphi_1 \parallel_{\beta} \leq 1$ and $\parallel \varphi_2 \parallel_{\beta} \leq 1$. By the positive homogeneity of $T_{\infty}$, this result can be extended to all $\varphi_1,\ \varphi_2 \in B_b(\mathbb{R}^d)$. 

Finally, we check the constant preserving property. For any constant $c$, $\varphi \in B_b(\mathbb{R}^d)$,
\begin{equation*}
  \begin{split}
    \vert T_{\infty}(c+\varphi)-T_{\infty}\varphi-c\vert \leq \vert T_{\infty}(c+\varphi)- T_{n}(c+\varphi)\vert +\vert T_{n}\varphi-T_{\infty}\varphi \vert.
  \end{split}
\end{equation*}
And $\frac{\vert T_{\infty}(c+\varphi)- T_{n}(c+\varphi)\vert}{\parallel\varphi+c\parallel_{\beta}}\leq \underset{\varphi:\parallel\varphi\parallel_{\beta}\neq 0}{{\rm sup}}\frac{\vert \bar{T}\varphi-\tilde{T}\varphi\vert}{\parallel\varphi\parallel_{\beta}}=\rho_{\beta}(T_n, T_{\infty})\rightarrow 0 $, $\vert T_{n}\varphi-T_{\infty}\varphi \vert \leq \rho_{\beta}(T_n, T_{\infty})\rightarrow 0$, as $n \to\infty$. So we have $ T_{\infty}(c+\varphi) = T_{\infty}\varphi+c$.

In summary, we prove that $T_{\infty}$ is an SE operator thus $\mathcal{M}$ is a complete metric space under $\rho_{\beta}$. By the relationship of $\rho_{\beta}$ and $d_{\beta}$, $\mathcal{M}$ is also a complete space under $d_{\beta}$.
\end{proof}

{
\renewcommand{\theassumption}{A}
\begin{assumption}
\label{Ass1}
There exist $\beta\geq0$ and $q\in[0,1)$ such that, for every ordered pair
$x,y\in {\Omega}$ and every $\mu\in\cP(x)$, there is a directed admissible coupling
$(\nu,\pi)$ satisfying
\begin{equation}\label{eq:coupling-cost-contraction}
 \int_{{\Omega\times \Omega}}{d_{\beta}}(u,v)\,\pi(du,dv)
 \leq q\,{d_{\beta}}(x,y).
\end{equation}
For $x=y$, this is understood by taking $\nu=\mu$ and the diagonal coupling.
\end{assumption}

\begin{lemma}[Admissible coupling implies nonlinear contraction]
\label{thm:direct-coupling}
Under Assumption \ref{Ass1},
\begin{equation}\label{eq:operator-contraction}
 \normmm {T\varphi}_{\beta}
 \leq q\normmm {\varphi}_{\beta}
\end{equation}
for every $\varphi$ with finite weighted oscillation seminorm.
\end{lemma}
\begin{proof}
Fix $x,y\in{\Omega}$ and $\mu\in\cP(x)$.  Choose $(\nu,\pi)$ as in the coupling
assumption.  Since $\nu$ is an admissible
competitor at $y$,
\begin{align*}
 \mu(\varphi)-T\varphi(y)
 &\leq \mu(\varphi)-\nu(\varphi)\\
 &=\int_{{\Omega\times \Omega}}
       \bigl(\varphi(u)-\varphi(v)\bigr)\,\pi(du,dv)\\
 &\leq \normmm {\varphi}_{\beta}
       \int_{E\times E}d_{\beta}(u,v)\,\pi(du,dv)\\
 &\leq q\normmm {\varphi}_{\beta}d_{\beta}(x,y).
\end{align*}
Taking the supremum over $\mu\in\cP(x)$ gives
\[
 T\varphi(x)-T\varphi(y)
 \leq q\normmm {\varphi}_{\beta}d_{\beta}(x,y).
\]
The assumption is imposed on ordered pairs, so the same argument with $x$
and $y$ interchanged gives the reverse one-sided estimate.  Division by
$d_{\beta}(x,y)$ and supremum over $x\neq y$ prove
\eqref{eq:operator-contraction}.  Notice that no maximizing measure or
control is used.
\end{proof}
}

Denote $(\bar{T}^k)^*T=T\bar{T}^k$, for any $T,\bar{T} \in \mathcal{M}$.
\begin{lemma}\label{anewlemma}
    Under the same conditions as in Proposition~\ref{thm:direct-coupling}, there exists a constant $n \in \mathbb{N}$, such that for any $k \geq n$, any SE operators $\hat{T},\ \check{T} \in \mathcal{M}$,
    \begin{equation}
       d_{\beta}((\bar{T}^k)^*\hat{T},(\bar{T}^k)^*\check{T}) \leq {q}d_{\beta}(\hat{T},\check{T}).
    \end{equation}
\end{lemma}
\begin{proof}
  For any SE operators $\hat{T},\ \check{T}\in \mathcal{M}$, using an alternative representation $\varphi\geq 0$, we choose $n$ as in Proposition~\ref{thm:direct-coupling}, for any $k \geq n $,
\begin{equation*}
  \begin{split}
     d_{\beta}((\bar{T}^k)^*\hat{T},(\bar{T}^k)^*\check{T})
     =&\underset{\varphi:\normmm{\varphi}_{\beta}\neq 0}{{\rm sup}}\frac{\vert (\bar{T}^k)^*\hat{T}\varphi-(\bar{T}^k)^*\check{T}\varphi\vert}{\normmm{\varphi}_{\beta}}\\
      \leq   & {q}(\underset{\varphi:\normmm{\varphi}_{\beta}\neq 0}{{\rm sup}}\frac{\vert\hat{T}\bar{T}^k\varphi-\check{T}\bar{T}^k\varphi \vert}{\normmm{\bar{T}^k\varphi}_{\beta}})\\
     \leq   & {q}(\underset{\varphi:\normmm{\varphi}_{\beta}\neq 0}{{\rm sup}}\frac{\vert\hat{T}\varphi-\check{T}\varphi\vert}{\normmm{\varphi}_{\beta}})\\
     =  &{q}d_{\beta}(\hat{T},\check{T}).
  \end{split}
\end{equation*}
where ${q}$ is the same as the one in Proposition~\ref{thm:direct-coupling}. The third line follows from the result of Proposition~\ref{thm:direct-coupling}. We also have $\rho_{\beta}((\bar{T}^k)^*\hat{T},(\bar{T}^k)^*\check{T})\leq {q}\rho_{\beta}(\hat{T},\check{T})$ by the relationship of $\rho_{\beta}$ and $d_{\beta}$.
\end{proof}

\begin{theorem}\label{disth}
 If Assumption \ref{Ass1} holds, then there exists $T_{\infty} \in \mathcal{M}$, such that
 \begin{equation}
   \rho_{\beta}((\bar{T}^{k})^*T_{\delta_x},T_{\infty})\rightarrow 0\ \ as\ k\rightarrow \infty,
 \end{equation}
where $T_{\delta_x}\varphi=\varphi(x)$ for any $x \in \mathbb{R}^d$, $\varphi \in B_b(\mathbb{R}^d)$ and $(\bar{T}^k)^{*}T_{\delta_x}=T_{\delta_x}\bar{T}^k$.
\end{theorem}
\begin{proof}

According to Lemma \ref{anewlemma}, for any fixed $k\geq n$, consider subsequence that SE operators to the power of some integer multiples of k, i.e. $\bar{T}^{mk}$ and $\bar{T}^{lk}$, $m,\ l\in \mathbb{N}_{+}$ and $m > l$,
\begin{equation*}
 \rho_{\beta}((\bar{T}^{mk})^*T_{\delta_x},(\bar{T}^{lk})^*T_{\delta_x}) 
 \leq   {q}^l \rho_{\beta}((\bar{T}^{(m-l)k})^*T_{\delta_x},T_{\delta_x}).
\end{equation*}
Furthermore,
\begin{equation*}
  \begin{split}
     \rho_{\beta}((\bar{T}^{(m-l)k})^*T_{\delta_x},T_{\delta_x})
   = & \underset{\varphi:\parallel\varphi\parallel_{\beta}\leq 1}{{\rm sup}}\vert (\bar{T}^{(m-l)k}\varphi)(x)-\varphi(x)\vert \\
\leq & \vert 1+\beta V(x)+(\bar{T}^{(m-l)k}(1+\beta V))(x)\vert\\
\leq & 2+\beta K  \frac{1-\gamma^{(m-l)k}}{1-\gamma}+\beta V(x)+\beta\gamma^{(m-l)k}V(x) < \infty .\\
  \end{split}
\end{equation*}
Thus, $\rho_{\beta}((\bar{T}^{mk})^*T_{\delta_x},(\bar{T}^{lk})^*T_{\delta_x})  \rightarrow 0$, as $m,\ l\rightarrow\infty$. So, under the metric $\rho_{\beta}$,    $(\bar{T}^{mk})^*T_{\delta_x}$ is a Cauchy sequence in $\mathcal{M}$. Therefore, by the completeness of the space $\mathcal{M}$ under $\rho_{\beta}$, there exists $T_{\infty}\in \mathcal{M}$ such that $\rho_{\beta}((\bar{T}^{mk})^*T_{\delta_x},T_{\infty})\ \rightarrow\ 0\ as\ m \rightarrow \infty$.

Note that, generally speaking, $T_{\infty}$ may depend on the choice of $k$. To see that this is independent of $k$, we consider any different $k_1,\ k_2 \geq n$, $m \in \mathbb{N}_+$,
\begin{equation*}
  \begin{split}
     & \rho_{\beta}((\bar{T}^{mk_1})^*T_{\delta_x},(\bar{T}^{mk_2})^*T_{\delta_x}) 
    \leq  \underset{\varphi:\parallel\varphi\parallel_{\beta}\leq 1}{{\rm sup}}(\vert (\bar{T}^{mk_1}\varphi)(x)\vert+ \vert (\bar{T}^{mk_2}\varphi)(x) \vert)\\
    = & (1+\beta V(x))\underset{\varphi:\parallel\varphi\parallel_{\beta}\leq 1}{{\rm sup}}(\frac{\vert (\bar{T}^{mk_1}\varphi)(x)\vert}{1+\beta V(x)}+ \frac{\vert( \bar{T}^{mk_2}\varphi)(x)\vert}{1+\beta V(x)})\\
 \leq & (1+\beta V(x))(\underset{\varphi:\parallel\varphi\parallel_{\beta}\leq 1}{{\rm sup}}(\parallel\bar{T}^{mk_1}\varphi\parallel_{\beta}+\parallel\bar{T}^{mk_2}\varphi\parallel_{\beta}))\\
\leq  & 2{q}^m(1+\beta V(x))\parallel \varphi \parallel_{\beta}\ \rightarrow\ 0 \ as \  m\rightarrow \infty.
  \end{split}
\end{equation*}

Thus, $\bar{T}^{mk} \rightarrow T_{\infty}$ as $m \rightarrow \infty$ for any $k \geq n $, which implies that $T_{\infty}$ is independent of $k$. Note $\{mk: m \in \mathbb{N}, k \geq n\} = \{n, n+1,\ldots\}$, so $ \rho_{\beta}((\bar{T}^k)^*T_{\delta_x},T_{\infty})\rightarrow 0 \ \ \ as\ \ k\rightarrow\infty.$
\end{proof}
\begin{lemma}\label{Lemma3.7}
  If for any $x \in \mathbb{R}^d$, assume $\rho_{\beta}((\bar{T}^k)^*T_{\delta_x},T_{\infty})\rightarrow 0 \ as\  k\rightarrow\infty$, then $T_{\infty}\bar{T}= T_{\infty}$.
\end{lemma}
\begin{proof}
  For any $\varphi\in B_b(\mathbb{R}^d)$, $\varphi \geq 0$,
\begin{equation*}
(\bar{T}^{k+1})^*T_{\delta_x}\varphi=(\bar{T}^{k+1}\varphi)(x)\rightarrow T_{\infty}\varphi \ \ as\ \ k\rightarrow\infty ,
\end{equation*}
and
\begin{equation*}
     (\bar{T}^{k+1})^*T_{\delta_x}\varphi 
       = (\bar{T}^{k}(\bar{T}\varphi))(x)
       \rightarrow T_{\infty}\bar{T}\varphi \ \ as\ \ k\rightarrow\infty.
\end{equation*}

Thus, $T_{\infty}=T_{\infty}\bar{T}$. 
\end{proof}
\begin{theorem}
  Under Assumption \ref{Ass1}, SMS $\bar{T}^k$ has a unique ISE $T_{\infty} \in \mathcal{M}$, such that $T_{\infty}=T_{\infty}\bar{T}$.
\end{theorem}
\begin{proof}
We only need to prove the uniqueness.

Assume there exists another $\tilde{T}_{\infty}\in M$ such that $\tilde{T}_{\infty}=\tilde{T}_{\infty}\bar{T}$. Then we have
\begin{equation*}
  \rho_{\beta}(\tilde{T}_{\infty},T_{\infty}) = \rho_{\beta}(\tilde{T}_{\infty}\bar{T},T_{\infty}\bar{T}) 
        = \rho_{\beta}(\tilde{T}_{\infty}\bar{T}^k,T_{\infty}\bar{T}^k)
        \leq{q}\rho_{\beta}(\tilde{T}_{\infty},T_{\infty}).   
\end{equation*}
Here $k \geq n$ as $n$ appearing in Proposition 3.3. Note that $0<{q}<1$, then \\
$\rho_{\beta}(\tilde{T}_{\infty},T_{\infty})= 0$.
\end{proof}

Now we turn to the continuous case. In this content,
we consider a family of SEs parameterized by $t \in \mathbb{R}^{+}$:
\begin{equation}
    T_t:B_b(\mathbb{R}^d) \to B_b(\mathbb{R}^d),t\geq 0.
\end{equation}
We assume that $T_t$ is an SMS that satisfies the property in Definition \ref{deftt}.


Then we can make similar assumptions on $T_a$ for a certain $a>0$ and obtain similar results on the convergence of the discrete sublinear semigroup $T_a^n$. { However, verification of the nonlinear local Doublin condition is normally a challenging task. This motivates us to consider an alternative weaker condition in the following theorem. Then, we can} deduct the convergence of $T_t$ as $t \to \infty$ from the convergence of $T_a^n$.

\renewcommand{\theassumption}{$A^{\prime}$}
\begin{assumption}\label{Ass3}
  There exists a function $V: \mathbb{R}^d \rightarrow [0,+\infty]$ and constants $K \geq 0$, $\gamma \in (0,1)$, $a > 0$, such that for all $x\in \mathbb{R}^d$,
  \begin{equation*}\label{A3e}
    (T_aV)(x) \leq \gamma V(x)+K.
  \end{equation*}  
\end{assumption}

\renewcommand{\theassumption}{$A^{\prime\prime}$}
\begin{assumption}\label{Ass3.1}
  There exists a constant $\delta>0$ such that for the same function $V$ and the same constants $a>0$, $K \geq 0$, for all $x\in \mathbb{R}^d$,
  \begin{equation*}
    (T_tV)(x) \leq \delta V(x)+K,\ 0 \leq t \leq a.
  \end{equation*}
  \end{assumption}

\renewcommand{\theassumption}{$B^{\prime}$}
\begin{assumption}\label{Ass4}
  { For some
\begin{equation}\label{eq:R-condition}
 R>\frac{2K}{1-\gamma}
\end{equation}
there is $\varepsilon_R\in(0,1]$ such that, whenever
$V(x)+V(y)\leq R$, every $\mu\in\cP(x)$ has a directed admissible coupling
$(\nu,\pi)$ to $y$ satisfying
\begin{equation}\label{eq:local-coalescence}
 \pi\{(u,v):u=v\}\geq\varepsilon_R.
\end{equation}
}
\end{assumption}

\begin{remark}
    In Assumption \ref{Ass3}, the constant $\gamma$, $K$ may depend on the selected time $a$. Since '$a$' is a fixed constant, it will not affect the proof of the subsequent theorems. However, $\gamma$, $K$ must be independent of the initial value $x$. 
\end{remark}

\begin{theorem}\label{C3maintheorem}
 If Assumptions \ref{Ass3}, \ref{Ass3.1} and \ref{Ass4} hold, then there exists $T_{\infty} \in \mathcal{M}$, such that
 \begin{equation} \label{limitan}
   \rho_{\beta}((T_a^{*})^nT_{\delta_x},T_{\infty})\rightarrow 0\ \ as\ n\rightarrow \infty.
 \end{equation}
Furthermore, if Assumption \ref{Ass3.1} also holds, we can obtain that
 \begin{equation}\label{limittinf}
      \rho_{\beta}((T_t^*T_{\delta_x},T_{\infty})\rightarrow 0\ \ as\ t\rightarrow \infty.
 \end{equation}
\end{theorem}

\begin{proof} {  We first prove 
(\ref{limitan}). For this, we suppose \eqref{eq:upper-representation} and
Assumptions~{$A^{\prime}$} and {$B^{\prime}$} hold.  Choose
\begin{equation*}
 0<\beta<\frac{2\varepsilon_R}{\gamma R+2K}
\end{equation*}
and set
\begin{equation}\label{eq:q-choice}
 q=\max\left\{
 1-\varepsilon_R+\frac\beta2(\gamma R+2K),
 \frac{2+\beta(\gamma R+2K)}{2+\beta R}
 \right\}.
\end{equation}
We claim that (\ref{eq:coupling-cost-contraction})  holds with this
$(\beta,q)$, and consequently
\begin{equation}\label{eq:Harris-contraction}
 \normmm {T\varphi}_{\beta}\leq q\normmm {\varphi}_{\beta}.
\end{equation}
To assert the claim, put $S=V(x)+V(y)$. First, suppose $S\leq R$.  Fix
$\mu\in\cP(x)$ and use the coupling in
Assumption~\ref{Ass4}.  Notice that since both marginals are admissible, Assumption
\ref{Ass3} yields
\[
 \int V(u)\,\pi(du,dv)+\int V(v)\,\pi(du,dv)
 \leq\gamma S+2K.
\]
The cost vanishes on the diagonal.  Hence
\begin{align}
 \int d_{\beta}(u,v)\,d\pi
 &\leq 2\pi\{u\neq v\}
       +\beta\int\bigl(V(u)+V(v)\bigr)\,d\pi\notag\\
 &\leq2(1-\varepsilon_R)+\beta(\gamma R+2K)\notag\\
 &\leq
 \left[1-\varepsilon_R+\frac{\beta}{2}(\gamma R+2K)\right]
 d_{\beta}(x,y).
 \label{eq:inside-estimate}
\end{align}

Now suppose $S>R$.  Choose any $\nu\in\cP(y)$ and any coupling of
$\mu$ and $\nu$; the product coupling is sufficient.  The Lyapunov estimate
gives
\begin{equation*}
 \int d_{\beta}(u,v)\,d\pi
 \leq2+\beta(\gamma S+2K).
\end{equation*}
The function
\[
 s\longmapsto
 \frac{2+\beta(\gamma s+2K)}{2+\beta s}
\]
is decreasing because its derivative has the sign of
$\gamma-1-\beta K<0$.  Consequently,
\begin{equation}\label{eq:outside-estimate}
 \frac{\int d_{\beta}\,d\pi}{d_{\beta}(x,y)}
 \leq
 \frac{2+\beta(\gamma R+2K)}{2+\beta R}.
\end{equation}
The first entry of \eqref{eq:q-choice} is less than one by the choice of $\beta$; the second is less than one by
\eqref{eq:R-condition}.  Thus $q<1$.  Estimates
\eqref{eq:inside-estimate} and \eqref{eq:outside-estimate} prove
(\ref{eq:coupling-cost-contraction}), and
Proposition~\ref{thm:direct-coupling} proves
\eqref{eq:Harris-contraction}.
 Iterations and then applying the same argument as in the proof of Lemma \ref{anewlemma} give (\ref{limitan}).}

To prove (\ref{limittinf}), similar to Lemma \ref{anewlemma}, we can prove that for any $0\leq t_1, t_2<a$,
\begin{equation*}
    \rho_{\beta}(T_a^*T_{t_1}^*T_{\delta_x},T_a^*T_{t_2}^*T_{\delta_x}) \leq {q}  \rho_{\beta}(T_{t_1}^*T_{\delta_x},T_{t_2}^*T_{\delta_x}).
\end{equation*}
Moreover, by Assumption \ref{Ass3.1}, we can prove that $\rho_{\beta}(T_{t_1}^*T_{\delta_x},T_{t_2}^*T_{\delta_x})$ is bounded for any $0\leq t_1, t_2<a$. For this, there exists a constant $M>0$, such that 
\begin{equation*}
   \begin{split}
       &\rho_{\beta}(T_{t_1}^*T_{\delta_x},T_{t_2}^*T_{\delta_x}) 
    = \underset{\varphi:\parallel\varphi\parallel_{\beta}\leq 1}{{\rm sup}}\vert (T_{t_1}\varphi)(x)-(T_{t_2}\varphi)(x)\vert 
    \leq  2M(1+\beta V(x)).
   \end{split}
\end{equation*}

Let $na \leq t < (n+1)a$, then
\begin{equation*}
    \begin{split}
     & \rho_{\beta}((T_t^*T_{\delta_x},T_{\infty}) 
     \leq  \rho_{\beta}((T_t^*T_{\delta_x},(T_a^{*})^nT_{\delta_x})+\rho_{\beta}((T_a^{*})^nT_{\delta_x},T_{\infty})\\
     \leq & {q}^n \rho_{\beta}((T_{t-na}^*T_{\delta_x},T_{\delta_x})+\rho_{\beta}((T_a^{*})^nT_{\delta_x},T_{\infty})\\
     \leq & {q}^n 2M(1+\beta V(x))+\rho_{\beta}((T_a^{*})^nT_{\delta_x},T_{\infty})
     \to 0,\ as\ n\to \infty. 
    \end{split}
\end{equation*}
 Thus, we proved the convergence of $T_t$, as $t \to \infty$.   
\end{proof}

{
The local assumption is imposed before taking the supremum in $T$: every
first marginal $\mu\in\cP(x)$ must be coupled to some admissible second
marginal in $\cP(y)$.  The operator-level oscillation theorem only sees the
resulting estimate after both optimizations.  Thus
Theorem~\ref{C3maintheorem} retains both marginal admissibility and the
meeting probability, which are the distinctive outputs of reflection
coupling.}

\begin{lemma}\label{lemma 3.12}
  Assume $\rho_{\beta}((T_a^*)^nT_{\delta_x},T_{\infty})\rightarrow 0 \ as\  n\rightarrow\infty$, then $T_{\infty}T_a= T_{\infty}$.\\
  Assume $ \rho_{\beta}((T_t^*T_{\delta_x},T_{\infty})\rightarrow 0\ \ as\ t\rightarrow \infty$, then $T_{\infty}T_t= T_{\infty}$ for any fixed $0 \leq t < \infty$.
\end{lemma}
\begin{proof}
    The proof is similar as that of Lemma \ref{Lemma3.7}.
\end{proof}
\begin{theorem}\label{coninth}
  Under Assumption \ref{Ass3}, \ref{Ass3.1} and \ref{Ass4}, SMS $T_t$, $t\geq0$, has a unique ISE $T_{\infty} \in \mathcal{M}$, such that $T_{\infty}=T_{\infty}T_t$, for any $t \geq 0$.
\end{theorem}
\begin{proof}
    The result is derived from Theorem \ref{C3maintheorem} and Lemma \ref{lemma 3.12}.
\end{proof}

\section{ Existence of ISEs of G-diffusions }
In this section, we will use the previous theorems to prove the existence of ISEs of some $G\text{-}$diffusions. In order to find a probability measure such that Assumption \ref{Ass4} can be satisfied, we will need to extend the result in Denis et al. \cite{DenishU}. In \cite{DenishU}, they extended the domain of $G\text{-}$expectation by constructing an upper expectation $\bar{E}[\cdot]$,
\begin{equation}
    \bar{E}[X]\coloneqq \underset{P \in \mathcal{P}_G}{{\rm sup}}E_P[X],\ X \in L^0(\Omega_T),
\end{equation}
where $\mathcal{P}_G$ is a weakly compact family of probability measures that depends on a parameter related to $G\text{-}$normal distribution. This upper expectation coincides with the $G\text{-}$expectation on $L_G^1(\Omega_T)$.
In some special case, the $G$-expectation has a more explicit expression described below.

Let $ (\Omega, \mathcal{F}, P)$ be a probability space and $(W_t)_{t\geq 0}$ a d-dimensional Brownian motion in this space, $\mathcal{F}_t$ the filtration generated by $W$ and $\mathbb{F} = \{\mathcal{F}_t\}_{t\geq 0}$. We also denote another $\sigma\text{-}$filed for fixed $s,\ t\geq 0$, $ \mathcal{F}^{t,s}\coloneqq \sigma\{W_{t+u}-W_t, 0 \leq u \leq s\},\ \mathbb{F}^t\coloneqq \{\mathcal{F}^{t,s} \}_{s\geq 0}.$

Let $\Theta$ be a given bounded and closed subset in $\mathbb{R}^{d \times d}$, $\mathcal{A}^{\Theta}_{t,T}$ be the collection of all $\Theta$-valued $\mathbb{F}$-adapted process on an interval $[t,T] \in [0, \infty)$. For each fixed $\theta \in \mathcal{A}^{\Theta}_{t,T}$, we define a process $ B^{t,\theta}_{T} \coloneqq \int^T_t \theta_s d W_s$.
For each $n=1,2,...$, $\varphi \in C_{b,lip}(\mathbb{R}^{d \times n})$ and \\$0 \leq t_1,..., \leq t_n < \infty$, the $G\text{-}$expectation can be equivalently defined by
\begin{equation}
    \mathbb{E}[\varphi(B_{t_{1}},B_{t_{2}}-B_{t_{1}},...,B_{t_{n}}-B_{t_{n-1}})]=\underset{\theta \in  \mathcal{A}^{\Theta}_{0,\infty}}{{\rm sup}} E_P[\varphi(B^{0,\theta}_{t_1},B^{t_1,\theta}_{t_2},...,B^{t_{n-1},\theta}_{t_n})].
\end{equation}
Because the upper expectation $\bar{E}$ coincides with the $G$-expectation $\mathbb{E}$ in $L_G^1$, we will use the symbol $\bar{E}$ commonly in the rest of the paper.

According to this result, we can greatly reduce the calculation difficulty by transforming the calculations of $G$-expectations into those of linear expectations. In the subsequent parts of this paper, we will extend this result in two aspects. One is that we do not need $\varphi$ to be Lipschitz continuous. We only assume that $\varphi$ is bounded and measurable. The other point is that we will extend $\varphi(B_t)$ to $\varphi(X_t)$, where $X_t$ is a solution of following $G\text{-}$SDE, $ X_t = x+\int_0^t b(X_s)ds+B_t$,
where $x \in \mathbb{R}^d$, $b$ is a Lipschitz function when $\varphi \in C_{b,lip}(\mathbb{R}^{d})$,  and $b\in C_b^2$ when $\varphi \in B_b(\mathbb{R}^d)$.

\begin{subsection}{Invariant distribution of $G$-Brownian motion on unit circle}

Firstly, we consider the $G\text{-}$Brownian motion on the unit circle $S^1 = [0,2\pi]$, { equipped with its circular distance
$\dist$,} defined by $X(t)= x+B_t$ mod $2 \pi$, where $B$ is a one-dimensional $G\text{-}$Brownian motion such that $B_1$ has normal distribution $N(0,[\underline{\sigma}^2, \bar{\sigma}^2])$. Here, $\bar{\sigma}^2 \geq \underline{\sigma}^2$ are constants. We omit the description of $G\text{-}$framework, the reader who is interested in more details is referred to \cite{DenishU}, \cite{PengG} and \cite{pengbook}. 
{ For $\varphi \in C_{b,lip}(S^1)$, set
\begin{equation}\label{1BMT}
  u(t,x)= \bar{E}[\varphi(X(t))].
\end{equation}
}
Then $u$ is a viscosity solution of the following fully nonlinear PDE (\cite{pengbook}),
\begin{equation}\label{1BMPDE}
  \frac{\partial}{\partial t}u=\frac{1}{2}\bar{\sigma}^2u_{xx}^+ - \frac{1}{2}\underline{\sigma}^2u_{xx}^-,\ u|_{t=0}=\varphi,\  x\in S^1.
\end{equation}

Let $(\Omega,\mathcal{F},P)$ be a probability space and $(W_t)_{t\geq0}$ be a $1$-dimensional Brownian motion in this space. The filtration generated by $W$ is denoted by $\mathcal{F}_t \coloneqq \sigma\{W_{\mu}:0\leq \mu \leq t\}\vee \mathcal{N}$, $\mathcal{F}=\{\mathcal{F}_t\}_{t\geq0}$, where $\mathcal{N}$ is the collection of $P$-null subsets. Let  $\mathcal{A}^{[\underline{\sigma}^2,\bar{\sigma}^2]}_{s,t}$ be the set of all $\mathcal{F}-$adapted processes with values in $[\underline{\sigma}^2, \bar{\sigma}^2]$ on an interval $[s,t]\subset [0,\infty)$. 
{ Set for any function $\varphi \in B_{b}(S^1)$,
\begin{equation}\label{1BMudef}
T_t\varphi(x)=\underset{\theta^2\in\mathcal{A}_{0,t}^{[\underline{\sigma}^2,\bar{\sigma}^2]}}{{\rm sup}}E_P[\varphi(x+\int_{0}^{t}\theta_s dW_s\  mod \ 2\pi)].
\end{equation}

From Denis et al. \cite{DenishU}, we know that for any function $\varphi \in C_{b,lip}(S^1)$, $
   u(t,x)=T_t\varphi(x)$.}
In the following, we will give a lemma establishing that (\ref{1BMudef}) gives a viscosity solution of (\ref{1BMPDE}) for all $\varphi \in B_b(S^1)$. { This result also extends the result in \cite{pengbook},} where the continuity condition of $\varphi$ was also required. {The main idea here is to study (\ref{1BMudef}),} without linking to the viscosity solution of (\ref{1BMPDE}) in the first place. Using { the coupling}, a purely probabilistic and analytical method {that is made to work for fully nonlinear semigroups,} we can prove that $u(t,x){=T_t\phi(x)}$ is continuous in $x$ and $t$ as long as $t > 0$. This, together with the semigroup property and the viscosity solution result of \cite{krylov1986}, \cite{krylov1987} and \cite{pengbook}, implies that $u(t,x)$ is a viscosity solution of (\ref{1BMPDE}), which completes the extension.

{
Let $\Phi$ be the standard normal distribution function and $C^{\alpha}(S^1)$ be the space of $\alpha-$H\"older continuous functions. For $f\in C^{\alpha}(S^1)$, define
\[\|f\|_{C^\alpha}=\|f\|_\infty+
\sup_{x\ne y}|f(x)-f(y)|/\dist(x,y)^\alpha.\]
 For a bounded $\varphi$ in $B_b(S)$, where $S$ is a subset of a metric space, define 
$$\osc _{S}(\varphi)=\sup_{x\in S}\varphi(x)-\inf_{x\in S}\varphi(x).$$

\begin{lemma}\label{mainleammapart1} If $\underline{\sigma}^2>0$ and $\varphi\in B_b(S^1)$,
 then for $t>0$, $x,y\in S^1$,
\begin{equation}\label{eq:coupling-modulus}
 |T_t\varphi(x)-T_t\varphi(y)|
 \leq\osc(\varphi)
\left[2\Phi\left(\frac{\dist(x,y)}
{2\underline\sigma\sqrt t}\right)-1\right]
 \leq\frac{\osc(\varphi)\dist(x,y)}{\underline\sigma\sqrt{2\pi t}}.
\end{equation}
For every $\tau_0>0$ and
$0<\alpha<1$,
\begin{equation}\label{eq:smoothing}
 \|T_t\varphi\|_{C^\alpha(S^1)}
 \leq\left(1+
 \frac{2\pi^{1-\alpha}}{\underline\sigma\sqrt{2\pi\tau_0}}\right)
 \|\varphi\|_\infty,
 \qquad t\geq\tau_0.
\end{equation}
 Moreover, 
for $s,t\geq0$ and $\varphi\in B_b(S^1)$,
\begin{equation*}\label{eq:DPP}
 T_{s+t}\varphi=T_s(T_t\varphi)
\end{equation*}
i.e. $T$ is a semigroup, and $u(t,x)=T_t\varphi(x)$ is
jointly
continuous on $(0,\infty)\times S^1$ and is a viscosity solution of (\ref{1BMPDE}). 
\end{lemma}

\begin{proof}
Let $x,y\in S^1$, with
$h=y-x$ and $|h|=\dist(x,y)$. The case $h=0$ is immediate.
Fix any $\theta\in\mathcal A$ and define
\[
 M_s=\int_0^s\theta_r\,dW_r,\qquad \bar X_s=x+M_s,
 \qquad \tau=\inf\{s\geq0:M_s=h/2\}.
\] 
Define
\begin{equation*}\label{eq:reflected-driver}
 \tilde W_s=\begin{cases}
 -W_s,&s\leq\tau,\\
 W_s-2W_\tau,&s>\tau.
 \end{cases}
\end{equation*}
Then it is easy to see that
\[\tilde W_s=\int_0^s\bigl(-I_{\{r\leq\tau\}}+I_{\{r>\tau\}}\bigr)\,dW_r,
\]
so $\tilde W$ is a continuous martingale with quadratic variation $s$, and is
Brownian by L\'evy's characterization. Moreover, $\tilde W$ has the same augmented natural filtration as that of $W$. Now put
\[
 \bar Y_s=y+\int_0^s\theta_r\,d\tilde W_r.
\]
Direct calculation gives
\[
 \bar Y_s=\begin{cases}
 y-M_s,&s\leq\tau,\\
 y+M_s-2M_\tau,&s>\tau.
 \end{cases}
\]
On $\{\tau<\infty\}$ we have $M_\tau=h/2$, so
\begin{equation*}\label{eq:coupling-meeting}
 \bar X_s=\bar Y_s\quad\text{for every }s\geq\tau.
\end{equation*}
The circle reductions $X_s=\bar X_s\pmod{2\pi}$ and
$Y_s=\bar Y_s\pmod{2\pi}$ also agree after $\tau$.
By reflection principle, $Y$ has the law of an admissible strongly
controlled process started at $y$. Hence
\begin{equation*}\label{eq:Y-admissible}
 E[\varphi(Y_t)]\leq T_t\varphi(y).
\end{equation*}

It remains to bound the probability of not meeting. Write
\[
 A_s=\langle M\rangle_s=\int_0^s\theta_r^2\,dr
 \geq\underline\sigma^2s.
\]
The clock is continuous, strictly increasing and tends to infinity.
The Dambis--Dubins--Schwarz theorem gives a Brownian motion $\beta$ such
that $M_s=\beta_{A_s}$; see \cite{contmartingale}. If $h>0$, then
\[
 \{\tau>t\}
 =\left\{\sup_{0\leq u\leq A_t}\beta_u<h/2\right\}
 \subseteq
 \left\{\sup_{0\leq u\leq\underline\sigma^2t}\beta_u<h/2\right\}.
\]
The Brownian reflection principle yields
\begin{equation*}\label{eq:meeting-tail}
 \PP(\tau>t)
 \leq2\Phi\left(\frac{|h|}{2\underline\sigma\sqrt t}\right)-1.
\end{equation*} 
For $h<0$, apply the same argument to $-M$ to obtain the same estimate as above.

For an arbitrary bounded Borel $\varphi$, equality after meeting implies
\[
 \begin{aligned}
 E[\varphi(X_t)]
 &\leq E[\varphi(Y_t)]+\osc(\varphi)\PP(\tau>t)\\
 &\leq T_t\varphi(y)+\osc(\varphi)
 \left[2\Phi\left(\frac{|h|}{2\underline\sigma\sqrt t}\right)-1\right].
 \end{aligned}
\]
Take the supremum over $\theta$ and then interchange $x,y$.
This proves the first inequality in \eqref{eq:coupling-modulus}.
The second follows from
\[
 2\Phi(r)-1=\frac2{\sqrt{2\pi}}\int_0^r e^{-z^2/2}\,dz
 \leq\sqrt{\frac2\pi}\,r,\qquad r\geq0.
\]
Thus, $\|T_t\varphi\|_\infty\leq\|\varphi\|_\infty$,
$\osc(\varphi)\leq2\|\varphi\|_\infty$ and $\dist(x,y)\leq\pi$
give \eqref{eq:smoothing}. Next, it is easy to see that the semigroup property follows from the DPP argument as in the proof of Theorem \ref{thm:DPP}. 

Finally, for bounded Borel $\varphi$, fix $\varepsilon>0$. The function
$g=T_\varepsilon\varphi$ is Lipschitz from the above, independently of any PDE argument.
By the semigroup property,
\begin{equation*}\label{eq:restart}
 T_t\varphi=T_{t-\varepsilon}g,\qquad t\geq\varepsilon.
\end{equation*}
The continuous initial-data viscosity solution theory applies to the right side. Since
$\varepsilon>0$ is arbitrary, it gives joint continuity and the viscosity
equation on the whole positive-time region. More precisely, $a\to G(a):
 =\frac12\left[\overline\sigma^2a^+-\underline\sigma^2a^-\right], a\in\R$, is convex, independent of $(t,x)$, and uniformly elliptic with
constants $\underline\sigma^2/2$ and $\overline\sigma^2/2$. The parabolic
Evans--Krylov interior regularity theorem for convex uniformly parabolic
equations gives the asserted local classical regularity; see
\cite{krylov1987}. This theorem is applied only after continuity and the
viscosity-solution property have been established.
\end{proof}
}

The following result was proved by Feng and Zhao (\cite{Zhaonon}). But their proof cannot be generalized to general cases immediately. We present a new proof here using the coupling result of Lemma 4.1. This proof inspires us to find a proof to a result in a much more general case in the later part of the paper.

{
Monotonicity and constant preservation imply
\begin{equation}\label{eq:osc-nonexpansive}
 \inf f\leq T_tf\leq\sup f,
 \qquad \osc(T_tf)\leq\osc(f).
\end{equation}
}
\begin{theorem}\label{S1contraction}
    Let $T_t$ be the SMS defined by (\ref{1BMT}) with the $G\text{-}$Brownian motion $(B_t)_{t\geq0}$ on the unit circle $S^1=[0,2\pi]$. The distribution of $B_1$ is $N(0,[\underline{\sigma}^2, \bar{\sigma}^2])$, where $\bar{\sigma}^2\geq \underline{\sigma}^2 > 0$ are constants. Then the SMS $T_t$ has a unique ISE.
\end{theorem}
\begin{proof}
In order to prove this theorem, we only need to verify whether $T_t$ defined in (\ref{1BMT}) can satisfy Assumptions \ref{Ass3}, \ref{Ass3.1} and \ref{Ass4} 
and then apply Theorem \ref{C3maintheorem}. Verification of Assumptions \ref{Ass3}, \ref{Ass3.1} is obvious due to the boundedness of $S^1$. {It remains to verify Assumption \ref{Ass4}}. 

{
Fix an ordered pair $x,y\in S^1$.  The case $x=y$ is handled by the
diagonal coupling, so assume $x\neq y$.  Let
$\mu\in\cP_a(x)$ be arbitrary.  By definition of the representing family,
choose an admissible weak system and a volatility control producing a process
$X$ from $x$ such that
\[
 \cL(X_a)=\mu.
\]
Apply the reflection construction to this particular first control.  It
produces a process $Y$ from $y$ on the same filtered probability space {and $\tau^{x,y} =\inf\{s\ge0:\ X_s=Y_s\ (\mathrm{mod}\ 2\pi)\}$}. In
the one-dimensional weak formulation, if
\[
 H_s=-1_{\{s\leq\tau^{x,y}\}}
      +1_{\{s>\tau^{x,y}\}},
 \qquad
 \widetilde W_t=\int_0^tH_s\,dW_s,
\]
then $\langle\widetilde W\rangle_t=t$, so L\'evy's characterization shows
that $\widetilde W$ is Brownian with respect to the same filtration.  The
original volatility control remains progressive and admissible.  Therefore
the second marginal is an admissible circle law:
\begin{equation*}\label{eq:circle-second-marginal}
 \nu:=\cL(Y_a)\in\cP_a(y).
\end{equation*}
After $\tau^{x,y}$ the two equations use the same noise and the two processes
are coalesced.  Hence
\begin{equation*}\label{eq:circle-coalescence-event}
 \{\tau^{x,y}\leq a\}
 \subseteq\{X_a=Y_a\},
\end{equation*}
and consequently
\begin{equation}\label{eq:circle-mismatch}
 \PP\{X_a\neq Y_a\}
 \leq\PP\{\tau^{x,y}>a\}
 \leq\rho_a,
\end{equation}
{where $\rho_a\coloneqq 2\Phi\Bigl(\frac{\pi}{2\underline\sigma\sqrt a}\Bigr)-1$.}

Let
\[
 \pi=\cL(X_a,Y_a)\in\Gamma(\mu,\nu).
\]
For $\beta=0$, the cost \eqref{cost} is
\[
 d_{0}(u,v)=2 ( 1_{\{u\neq v\}}).
\]
Using \eqref{eq:circle-mismatch},
\begin{equation*}\label{eq:circle-cost-contraction}
  \int_{S^1\times S^1}d_{0}(u,v)\,\pi(du,dv)
 =2\PP\{X_a\neq Y_a\}\notag
 \leq2\rho_a
 =\rho_a d_{0}(x,y). 
\end{equation*}
The construction began with an arbitrary $\mu\in\cP_a(x)$ and works for
every ordered pair $x,y$.  Thus it verifies
Assumption~\ref{Ass4} for the time-$a$ operator with $
 \beta=0,\qquad q=\rho_a.$
Theorem~\ref{C3maintheorem} now gives
\[
 \normmm {T_a\varphi}_{0}
 \leq\rho_a\normmm {\varphi}_{0}.
\]
Since 
${\normmm {\varphi}_{0}=\frac12\osc(\varphi)}$, this is precisely
\[
 \osc (T_a\varphi)
 \leq\rho_a\osc (\varphi).
\]
Iteration yields
\begin{equation}\label{eq:circle-skeleton-oscillation}
 \osc(T_{na}\varphi)
 \leq\rho_a^n\osc(\varphi).
\end{equation}
Theorem~\ref{C3maintheorem}, in its $\beta=0$ form, gives a unique
$T_a$-invariant sublinear expectation $\widetilde\cE$.  It is invariant for
the full semigroup.  Indeed, for $0\leq r<a$,
$\widetilde\cE T_r$ is also $T_a$-invariant because
\[
 (\widetilde\cE T_r)T_a
 =\widetilde\cE T_{r+a}
 =(\widetilde\cE T_a)T_r
 =\widetilde\cE T_r.
\]
Uniqueness therefore gives
$\widetilde\cE T_r=\widetilde\cE$.

Finally, monotonicity and preservation of constants imply
\[
 \inf_{S^1}{\varphi}
 \leq T_r{\varphi}(x)
 \leq\sup_{S^1}{\varphi},
\]
and hence
$\osc(T_r{\varphi})\leq\osc({\varphi})$.  For $t=na+r$, the semigroup property and
\eqref{eq:circle-skeleton-oscillation} give
\[
 \osc(T_t\varphi)
 =\osc\bigl(T_{na}(T_r\varphi)\bigr)
 \leq\rho_a^n\osc(T_r\varphi)
 \leq\rho_a^n\osc(\varphi).
\]
Moreover,
$\widetilde\cE[\varphi]=\widetilde\cE[T_t\varphi]$ lies between the minimum
and maximum of $T_t\varphi$.  The asserted uniform convergence estimate
follows.
}
\end{proof}



{
It is worth noting that for the linear Brownian motion on $S^1$, its invariant measure is a Lebesgue, but the invariant sublinear expectation of the G-Brownian motion on $S^1$ is genuinely nonlinear. This can be as follows. 

Let $\varphi\in C^{2,\alpha}(S^1)$, $0<\alpha<1$, and
$u(t,x)=T_t\varphi(x)$. 
Lemma \ref{mainleammapart1} gives the classical regularity at positive
times. Periodicity implies
\[
 \int_0^{2\pi}u_{xx}\,dx=0,\qquad
 \int_0^{2\pi}(u_{xx})^+\,dx
 =\int_0^{2\pi}(u_{xx})^-\,dx.
\]
Therefore
\begin{equation}\label{eq:mean-derivative}
 \frac{d}{dt}m(u(t,\cdot))
 =\frac1{4\pi}\int_0^{2\pi}
 \left[\overline\sigma^2(u_{xx})^+
       -\underline\sigma^2(u_{xx})^-\right]\,dx=\frac{\overline\sigma^2-\underline\sigma^2}{4\pi}
 \int_0^{2\pi}(u_{xx})^+\,dx.
 \end{equation}

We justify the initial derivative directly from the generator. For every
admissible control, It\^o's formula yields
\[
 E[\varphi(X_t^{x,\theta})]-\varphi(x)
 =E\int_0^t\frac12\theta_s^2\varphi_{xx}(X_s^{x,\theta})\,ds.
\]
Since $\varphi_{xx}$ is $\alpha$-H\"older and
\[
 E[\dist(X_s^{x,\theta},x)^\alpha]
 \leq E\left[\left|\int_0^s\theta_r\,dW_r\right|^\alpha\right]
 \leq\overline\sigma^\alpha s^{\alpha/2},
\]
uniformly in $x,\theta$, it follows that
\begin{equation*}\label{eq:Ito-generator-error}
 \left|E\int_0^t\frac12\theta_s^2
 [\varphi_{xx}(X_s^{x,\theta})-\varphi_{xx}(x)]\,ds\right|
 \leq C_\varphi t^{1+\alpha/2}.
\end{equation*}
The inequality
$\frac12\theta_s^2\varphi_{xx}(x)\leq G(\varphi_{xx}(x))$
gives
\[
 T_t\varphi(x)-\varphi(x)
 \leq tG(\varphi_{xx}(x))+C_\varphi t^{1+\alpha/2}.
\]
For the reverse bound choose, for each fixed initial point $x$, the
constant control
\[
 \theta_s=\begin{cases}
 \overline\sigma,&\varphi_{xx}(x)\geq0,\\
 \underline\sigma,&\varphi_{xx}(x)<0.
 \end{cases}
\]
This is allowed because the supremum defining the value is taken
separately at each initial point. It gives
\[
 T_t\varphi(x)-\varphi(x)
 \geq tG(\varphi_{xx}(x))-C_\varphi t^{1+\alpha/2}.
\]
Consequently,
\begin{equation}\label{eq:generator-limit}
 \left\|\frac{T_t\varphi-\varphi}{t}
 -G(\varphi_{xx})\right\|_\infty
 \leq C_\varphi t^{\alpha/2}\longrightarrow0.
\end{equation}
Integration against normalized Lebesgue measure now proves
\[
 \left.\frac{d}{dt}\right|_{0+}m(T_t\varphi)
 =m(G(\varphi_{xx}))
 =\frac{\overline\sigma^2-\underline\sigma^2}{4\pi}
 \int_0^{2\pi}(\varphi_{xx})^+\,dx.
\]

Suppose first that $\Delta:=\overline\sigma^2-\underline\sigma^2>0$.
For $0<\varepsilon<T$, integrate the classical identity only away from zero:
\begin{equation*}\label{eq:mean-epsilon}
 m(T_T\varphi)-m(T_\varepsilon\varphi)
 =\frac\Delta{4\pi}\int_\varepsilon^T\!\int_0^{2\pi}
 (u_{xx}(t,x))^+\,dx\,dt.
\end{equation*}
Uniform initial continuity and monotone convergence allow
$\varepsilon\downarrow0$. The resulting partial integrals are bounded
because $m(T_T\varphi)$ lies between $\inf\varphi$ and $\sup\varphi$.
Letting $T\to\infty$ and using Theorem \ref{S1contraction} give
\begin{equation}\label{eq:limit-formula}
 \Et_G[\varphi]=m(\varphi)
 +\frac{\overline\sigma^2-\underline\sigma^2}{4\pi}
 \int_0^\infty\!\int_0^{2\pi}(u_{xx}(t,x))^+\,dx\,dt,
\end{equation}
and finiteness of the integral.

If $\Delta=0$, write $q=\underline\sigma^2=\overline\sigma^2$.
The semigroup is the ordinary heat semigroup $P_t^q$ and preserves $m$.
To justify separately the finiteness assertion, note that
$u_{xx}=P_t^q\varphi_{xx}$, so
$\|u_{xx}(t)\|_\infty\leq\|\varphi_{xx}\|_\infty$ for $0<t\leq1$.
For $t\geq1$, the Fourier series of the heat kernel gives
\[
 \|u_{xx}(t)\|_\infty
 \leq2\|\varphi\|_\infty\sum_{n=1}^\infty n^2e^{-qn^2t/2}
 \leq C_q\|\varphi\|_\infty e^{-qt/2}.
\]
Thus the space--time integral is finite also in the linear case, and
\eqref{eq:limit-formula} reduces to $\Et_G[\varphi]=m(\varphi)$.

Now suppose $\underline\sigma<\overline\sigma$, and let
$\varphi\in C^{2,\alpha}(S^1)$ be nonconstant. Periodicity shows that
$\varphi_{xx}\not\equiv0$ and $\int\varphi_{xx}=0$, whence
$
 \int_0^{2\pi}(\varphi_{xx})^+\,dx>0.
$
The right derivative established in \eqref{eq:generator-limit} gives
$m(T_t\varphi)>m(\varphi)$ for all sufficiently small $t>0$.
Equation~\eqref{eq:mean-derivative} makes the mean nondecreasing thereafter.
Its limit is $\Et_G[\varphi]$, so
$m(\varphi)<\Et_G[\varphi]$.
Applying this argument to $-\varphi$ proves the other strict inequality. Thus we have \begin{equation}\label{eq:strict-mean}
 -\Et_G[-\varphi]<m(\varphi)<\Et_G[\varphi].
\end{equation}
If $\Et_G$ were linear, its upper and lower values would coincide, which
contradicts \eqref{eq:strict-mean}. Hence, it is nonlinear.

}

\end{subsection}

\begin{subsection}{Invariant distribution of $G\text{-}$SDEs}
    In this subsection, we consider a class of d-dimensional $G\text{-}$SDEs,
\begin{equation}\label{GSDE}
    X_{t}=x+\int_0^t b(X_s)ds+B_t, \ 0\leq t<\infty,
\end{equation}
where $x \in \mathbb{R}^d$ is the initial value, $b:\mathbb{R}^d \mapsto \mathbb{R}^d$, 
$B_t$ is the d-dimensional $G\text{-}$Brownian motion. However, for some reason, we will consider the process $X_T^{t,x}$ first,


\begin{equation}\label{barxTtnew}
  X_T^{t,x}=x+\int_t^Tb(X_s^{t,x})ds+B_T-B_t,\ 0\leq T <
\infty, 0 \leq t <T.
\end{equation} 
Then the corresponding result of (\ref{GSDE}) can be derived from that of (\ref{barxTtnew}). This approach is adopted to facilitate the application of the backward Kolmogorov equation in Theorem \ref{revth}. Moreover, this choice of time interval is also consistent with the proof of Lemma \ref{mianlemma}, which will be detailed in Section 5.

Usually, in the $G\text{-}$framework, SMS is typically defined as follows (\cite{Peng2005}, \cite{pengbook}),
\begin{equation}\label{barTv}
    \bar{T}_{T,t}\varphi(x)=\bar{\nu}(t,x)=\bar{E}[\varphi(X_T^{t,x})],\ \varphi \in C_{b,lip}(\mathbb{R}^d).
\end{equation}
Then $\bar{\nu}(t,x)$ is a viscosity solution of (\ref{PDE}) in Section 5. For the reader's convenience, we restate it as follows,
  \begin{equation}\label{mianpdeinap}
      \left\{\begin{aligned}
          & \partial_t\bar{\nu}+F(D^2\bar{\nu},D\bar{\nu},\bar{\nu},x)=0,\ 0\leq t\leq T,\\
          &\bar{\nu}(T,x)=\varphi(x), \ \varphi \in C_{b,lip}(\mathbb{R}^d),
      \end{aligned} \right. 
    \end{equation}
where $  F(D^2\bar{\nu},D\bar{\nu},\bar{\nu},x)=G(D^2\bar{\nu})+\left\langle D\bar{\nu},b(x)\right\rangle.$


Similarly, as in the case of $G$-Brownian motions on the unit circle,  we intend to represent $\bar{\nu}(t,x)$ as the supremum of a family of linear expectations of SDEs driven by linear Brownian motions. This extends the result presented in Section 3 of Denis et al. \cite{DenishU}. Recall $\Theta$ is a given bounded and closed subset in $\mathbb{R}^{d \times d}$, $\mathcal{A}^{\Theta}_{t,T}$ is the collection of all $\Theta$-valued $\mathbb{F}$-adapted process on an interval $[t,T]$. Here $\mathbb{F}$ can be similarly defined as in Section 4.1 for dimension $d$.

\begin{lemma}\label{mianlemma}
Let $X_T^{t,x}$ be defined by (\ref{barxTtnew}), where $b$ is a Lipschitz continuous function with Lipschitz constant ${L_b}$. Then for each $\varphi \in C_{b,lip}(\mathbb{R}^{d })$, for any fixed $0\leq T <
\infty, 0 \leq t \leq T $,  the $G\text{-}$expectation of $\varphi(X_T^{t,x})$, can be equivalently defined by
\begin{equation}\label{lemma4.5result}
    \bar{E}[\varphi(X_T^{t,x})]=\underset{\theta \in  \mathcal{A}^{\Theta}_{t,T}}{{\rm sup}} E_P[\varphi(x+\int^T_t b(X_s^{\theta})ds+\int_t^T\theta_sdW_s)],
\end{equation}
where $ X_t^{\theta}=x + \int_0^t b(X_s^{\theta})ds+\int_0^t\theta_sdW_s$.
\end{lemma}

The proof of this lemma is long, so we postpone it to Section 5 in order not to interrupt the main flow of this section.

\begin{remark}\label{remark}
    From the proof of Lemma \ref{mianlemma}, we can know the relationship between the set $\Theta$ and $G\text{-}$Brownian motion, that is,  $\underset{\gamma \in \Theta}{{\rm sup}}\ tr(\gamma \gamma^*)$ and $\underset{\gamma \in \Theta}{{\rm inf}}\ tr(\gamma \gamma^*)$ are the bounds of variance of the $G\text{-}$Brownian motion, i.e.
    \begin{equation*}
      \underset{\gamma \in \Theta}{{\rm inf}}\ tr(\gamma \gamma^*)= -\bar{E}[-B_1^2]\coloneqq \underline{\sigma}^2,\ \underset{\gamma \in \Theta}{{\rm sup}}\ tr(\gamma \gamma^*)=\bar{E}[B_1^2]\coloneqq\bar{\sigma}^2.  
    \end{equation*}

\end{remark}

Now we extend both the definition of SMS $\bar{T}_{t,T}$ and the result in Lemma \ref{mianlemma} to any $\varphi \in B_{b}(\mathbb{R}^d)$. For that, we give three assumptions,\\
\textbf{(H1)}\label{eq:H1} For  any element $\gamma$ in the given bounded set $\Theta$, for any $\xi \in \mathbb{R}^d$, there exists a positive constant $\kappa_0\geq 1$, such that $\kappa_0^{-1}\vert \xi \vert^2\leq \langle \gamma\gamma^*\xi,\xi \rangle\leq \kappa_0\vert  \xi \vert^2$.\\
\textbf{(H2)} $b \in C_b^2$.\\
\textbf{(H3)} There exist two positive constants $\lambda$ and $C$, such that for any $x \in \mathbb{R}^d$, $\langle x,b(x)\rangle \leq -\lambda\vert x\vert^2+C$.

\begin{remark}
  Under \textbf{(H1)}, we can deduce that for any $\gamma \in \Theta$, $\vert \gamma \vert>0$, no row of $\gamma$ can be a zero vector. It follows that for any $1 \leq i \leq d$, $\gamma^i (\gamma^i)^{T}>0$, where $\gamma^i$ is the $ith$ row of matrix $\gamma$, and then we have $-\bar{E}[-(B_1^i)^2]=\underline{\sigma_i}^2 =\underset{\gamma \in \Theta}{{\rm inf}} \gamma^i (\gamma^i)^T >0$, where $B^i$ is the $ith$ component of $B$. That is to say, \textbf{(H1)} implies that any component of $G\text{-}$Brownian motion $B$ is non-degenerate. The variance bounds of the $G\text{-}$Brownian motion $\underline{\sigma}^2$, $\bar{\sigma}^2$ can be determined by $\kappa_0$ in \textbf{(H1)}.
\end{remark}



{
Now we turn to the consideration of $X_t$, $t\geq 0$, with assumptions \textbf{(H1)}, \textbf{(H2)} and \textbf{(H3)}, to the point of considering ${\bar{T}_{T,t}}$ again with explicit notice. The relevant SMS is defined as follows 
\begin{equation}\label{bbbarTv}
    \bar{T}_{t}\varphi(x)=\tilde{\nu}(t,x)=\bar{E}[\varphi(X_{t})],\ \varphi \in B_b(\mathbb{R}^d).
\end{equation} 
We will prove that $(t,s)\mapsto \bar{T}_{t}\varphi(x)$ is continuous  in $(t,x)$. 
}

The proof will be provided by several lemmas in the following. 
Recall that { a \emph{weak admissible system} is a tuple
\[
 \mathfrak a=(\Omega,\FF,(\FF_t)_{t\geq0},\PP,W,\theta)
\]
such that the filtration satisfies the usual conditions, $W$ is a
$d$-dimensional $(\FF_t)$-Brownian motion, and $\theta$ is an
$(\FF_t)$-progressively measurable $\Theta$-valued process.} 

{
\begin{lemma}\label{lem:RA}
Let $\mathfrak a=(\Omega,\FF,(\FF_t),\PP,W,\theta)$ be a weak admissible
system and let $H$ be an $(\FF_t)$-predictable orthogonal matrix process.
Then
\[
 \widetilde W_t=\int_0^tH_s\,dW_s
\]
is a $d$-dimensional $(\FF_t)$-Brownian motion.  Consequently,
\[
 \widetilde{\mathfrak a}
 =(\Omega,\FF,(\FF_t),\PP,\widetilde W,\theta)
\]
is again a weak admissible system.
\end{lemma}

\begin{proof}
The process $\widetilde W$ is a continuous $(\FF_t)$-local martingale and
\[
 \langle\widetilde W\rangle_t
 =\int_0^tH_sH_s^*\,ds=tI_d.
\]
L\'evy's characterization therefore shows that $\widetilde W$ is Brownian
with respect to the same filtration.  The process $\theta$ was already
$(\FF_t)$-progressive and remains $\Theta$-valued, which proves weak
admissibility.  Notice that no assertion about the natural filtration of
$\widetilde W$ is required.
\end{proof}

Fix $x,y\in\R^d$ and a weak admissible system, denoted for short by
$\theta\in\mathcal A^\Theta$.  Let $X$ solve
\eqref{GSDE} from $x$ on that system.  We construct a second
process $Y$ from $y$ on the same filtered probability space.  Before the
meeting time
\begin{equation}\label{eq:meeting-time}
 \tau=\inf\{t\geq0:X_t=Y_t\},
\end{equation}
put
$
 Z_t=X_t-Y_t, r_t=|Z_t|,
 e_t=\frac{Z_t}{|Z_t|}.
$
By \eqref{eq:H1}, $\theta_t$ is invertible.  Define
\begin{equation*}\label{eq:v-H}
 v_t=\frac{\theta_t^{-1}e_t}{|\theta_t^{-1}e_t|},
 \qquad H_t=I-2v_tv_t^*,
 \qquad t<\tau,
\end{equation*}
and put $H_t=I$ after $\tau$.  Since $H_tH_t^*=I$,
\begin{equation}\label{eq:reflected-W}
 \widetilde W_t=\int_0^tH_s\,dW_s
\end{equation}
is Brownian by L\'evy's characterization.  Let
\begin{equation}\label{eq:coupled-Y}
 dY_t=b(Y_t)\,dt+\theta_t\,d\widetilde W_t,
 \qquad Y_0=y,
\end{equation}
and coalesce the two solutions after $\tau$.  Lemma~\ref{lem:RA} shows that
$(\Omega,\FF,(\FF_t),\PP,\widetilde W,\theta)$ is a weak admissible
system.  Thus $Y$ is an admissible controlled diffusion starting from $y$;
no natural-filtration transfer is needed.

The reflection matrix is singular only as a formula at $X=Y$.  One may
construct the coupled system up to the first meeting time, set $H=I$ after
meeting and then use pathwise uniqueness for the coalesced equation.  A
fully weak construction can alternatively be obtained by smoothing the
reflection field near the diagonal and passing to the limit; this is the
standard reflection-coupling construction used in Eberle {\cite{Eberle2016}}.

Let
\begin{equation}\label{eq:beta}
 \beta_t=\int_0^t v_s^*\,dW_s
\end{equation}
before $\tau$, extended after $\tau$ in any predictable unit direction.
Then $\beta$ is one-dimensional Brownian motion.  From
$I-H_t=2v_tv_t^*$ and
$
 \theta_tv_t=\frac{e_t}{|\theta_t^{-1}e_t|}$,
we obtain, for $t<\tau$,
\begin{equation}\label{eq:Z-SDE}
 dZ_t=\bigl(b(X_t)-b(Y_t)\bigr)\,dt
 +\frac{2e_t}{|\theta_t^{-1}e_t|}\,d\beta_t.
\end{equation}
The martingale part is parallel to $Z_t$, so the transverse It\^o
correction in the equation for $|Z_t|$ vanishes.  Consequently,
\begin{equation}\label{eq:r-SDE}
 dr_t=\mu_t\,dt+\alpha_t\,d\beta_t,
 \qquad
 \mu_t=\langle e_t,b(X_t)-b(Y_t)\rangle,
 \qquad
 \alpha_t=\frac2{|\theta_t^{-1}e_t|}.
\end{equation}
The drift and volatility satisfy
\begin{equation}\label{eq:radial-bounds}
 \mu_t\leq L_br_t,
 \qquad
 \frac2{\sqrt{\kappa_0}}\leq\alpha_t\leq2\sqrt{\kappa_0}.
\end{equation}
Indeed,
$|b(X)-b(Y)|\leq L_b|X-Y|$, and \eqref{eq:H1} gives
\[
 \kappa_0^{-1}
 \leq |\theta_t^{-1}e_t|^2
 =\langle(\theta_t\theta_t^*)^{-1}e_t,e_t\rangle
 \leq\kappa_0.
\]

Define
\begin{equation*}\label{eq:I-L}
 I_{L_b}(t)=\int_0^t e^{-2L_bs}\,ds
 =\begin{cases}
 \dfrac{1-e^{-2L_bt}}{2L_b},&L_b>0,\\[1.2ex]
 t,&L_b=0.
 \end{cases}
\end{equation*}
We have the following uniform meeting-time estimate.

\begin{lemma}\label{lem:meeting-tail}
For the reflection coupling above and every $t>0$,
\begin{equation*}\label{eq:meeting-tail}
 \PP(\tau>t)
 \leq
 2\Phi\left(
 \frac{\sqrt{\kappa_0}|x-y|}{2\sqrt{I_{L_b}(t)}}
 \right)-1,
\end{equation*}
uniformly over $\theta\in\mathcal A^\Theta$.
\end{lemma}

\begin{proof}
For $s<\tau$, apply the integrating factor $e^{-L_bs}$ to
\eqref{eq:r-SDE}:
\begin{equation*}\label{eq:integrating-factor}
 d(e^{-L_bs}r_s)
 =e^{-L_bs}(\mu_s-L_br_s)\,ds
  +e^{-L_bs}\alpha_s\,d\beta_s.
\end{equation*}
The finite-variation term is nonpositive by
\eqref{eq:radial-bounds}.  Extend $\alpha$ after $\tau$ while retaining
the bounds in \eqref{eq:radial-bounds}, and set
\[
 N_s=\int_0^s e^{-L_bu}\alpha_u\,d\beta_u.
\]
On $\{\tau>t\}$, for every $s\leq t$,
\[
 0<e^{-L_bs}r_s\leq|x-y|+N_s.
\]
Hence
\begin{equation}\label{eq:event-inclusion}
 \{\tau>t\}
 \subseteq
 \left\{\inf_{0\leq s\leq t}N_s>-|x-y|\right\}.
\end{equation}
Moreover,
\begin{equation}\label{eq:N-clock}
 \langle N\rangle_t
 =\int_0^t e^{-2L_bs}\alpha_s^2\,ds
 \geq\frac4{\kappa_0}I_{L_b}(t).
\end{equation}
By the Dambis--Dubins--Schwarz theorem, $N_s=\widehat\beta_{\langle
N\rangle_s}$ for a Brownian motion $\widehat\beta$; see
\cite{contmartingale}.  Since the quadratic-variation clock is continuous,
\eqref{eq:event-inclusion} and \eqref{eq:N-clock} imply
\[
 \{\tau>t\}
 \subseteq
 \left\{
 \inf_{0\leq u\leq4I_{L_b}(t)/\kappa_0}
 \widehat\beta_u>-|x-y|
 \right\}.
\]
The Brownian reflection principle gives the desired result of this lemma.
No independence between the DDS Brownian motion and its clock is used.
\end{proof}

The meeting time distribution is then used to prove the following positive-time Lipschitz regularisation result. 
\begin{proposition}
\label{prop:regularity}
Under the Conditions (H1) and (H2), for every $t>0$,
$x,y\in\R^d$ and $\varphi\in B_b(\R^d)$,
\begin{align}
 |T_t\varphi(x)-T_t\varphi(y)|
 &\leq\osc(\varphi)
 \left[
 2\Phi\left(
 \frac{\sqrt{\kappa_0}|x-y|}{2\sqrt{I_{L_b}(t)}}
 \right)-1
 \right]
 \label{eq:coupling-modulus-1}\\
 &\leq 
 \sqrt{\frac{\kappa_0}{2\pi I_{L_b}(t)}}
 \osc(\varphi)|x-y|.
 \label{eq:Lipschitz-bound}
\end{align}
In particular,
\begin{equation}\label{eq:Bb-Lip}
 T_t:B_b(\R^d)\longrightarrow\Lip(\R^d),
 \qquad
 [T_t\varphi]_{\Lip}
 \leq
 \sqrt{\frac{\kappa_0}{2\pi I_{L_b}(t)}}\osc(\varphi).
\end{equation}
\end{proposition}

\begin{proof}
Fix $t>0$ and $x,y\in\R^d$.  The case $x=y$ is immediate.  Choose an
arbitrary weak admissible system $\theta\in\mathcal A^\Theta$ for the
process $X=X^{x,\theta}$ starting from $x$, and construct the reflected
process $Y$ starting from $y$ as in (\ref{eq:coupled-Y}).  By
construction,
\begin{equation}\label{eq:coalescence-regularity}
 X_s=Y_s,\qquad s\geq\tau.
\end{equation}
Lemma~\ref{lem:RA} makes
$(\Omega,\FF,(\FF_t),\PP,\widetilde W,\theta)$ a weak admissible system.
Consequently, the marginal law of $Y$ is an admissible competitor in the
value function starting from $y$, and
\begin{equation}\label{eq:Y-admissible-Rd}
 E[\varphi(Y_t)]\leq T_t\varphi(y).
\end{equation}

Since $X_t=Y_t$ on $\{\tau\leq t\}$, every bounded Borel function satisfies
\[
 \varphi(X_t)-\varphi(Y_t)
 =\bigl(\varphi(X_t)-\varphi(Y_t)\bigr)I_{\{\tau>t\}}.
\]
Moreover,
$\varphi(\xi)-\varphi(\eta)\leq\osc(\varphi)$ for all
$\xi,\eta\in\R^d$.  Therefore
\begin{align*}
 E[\varphi(X_t)]
 &\leq E[\varphi(Y_t)]
      +\osc(\varphi)\PP(\tau>t)\\
 &\leq T_t\varphi(y)
      +\osc(\varphi)\PP(\tau>t).
\end{align*}
Applying Lemma~\ref{lem:meeting-tail} and taking the supremum over the original systems in $\mathcal A^\Theta$ give
\[
 T_t\varphi(x)-T_t\varphi(y)
 \leq\osc(\varphi)
 \left[
 2\Phi\left(
 \frac{\sqrt{\kappa_0}|x-y|}{2\sqrt{I_{L_b}(t)}}
 \right)-1
 \right].
\]
No maximizing control is required: the estimate holds before taking the
supremum for every admissible system.  Repeating the argument with $x$ and
$y$ interchanged and combining the two one-sided inequalities proves
\eqref{eq:coupling-modulus-1}.

Finally, for $z\geq0$,
\[
 2\Phi(z)-1
 =\frac2{\sqrt{2\pi}}\int_0^z e^{-r^2/2}\,dr
 \leq\sqrt{\frac2\pi}\,z.
\]
Substitution of
$z=\sqrt{\kappa_0}|x-y|/(2\sqrt{I_{L_b}(t)})$ proves
\eqref{eq:Lipschitz-bound}.  Dividing both sides of \eqref{eq:Lipschitz-bound} by $|x-y|$ and taking the supremum over $x\ne y$ prove
\eqref{eq:Bb-Lip}.
\end{proof}

We now prove the semigroup identity used later in the note.  The only structural facts
needed are that weak admissible laws can be conditioned at a deterministic
time and can be pasted with a measurably selected continuation law.  We spell
out those facts because, for bounded Borel terminal functions, merely writing
``condition at time $s$'' does not by itself justify the reverse dynamic
programming inequality.

It is convenient to use a canonical-law version of the weak formulation.
Let $\mathbb U$ be the standard Borel space of equivalence classes (modulo
Lebesgue-null sets) of measurable maps
$\alpha:\R_+\to\Theta$, with the topology of local convergence in measure.
On
\[
 \Omega^c=C(\R_+;\R^d)\times C(\R_+;\R^d)\times\mathbb U
\]
write $(X,W,\alpha)$ for the coordinate processes and use the canonical
filtration generated by $X$, $W$ and the restrictions of $\alpha$ to past
time intervals.  Denote by $\mathfrak P(x)$ the laws on $\Omega^c$ induced by
all weak admissible systems starting from $x$.  Equivalently, under
$P\in\mathfrak P(x)$, $W$ is a Brownian motion, $X_0=x$, and
\[
 M_r:=X_r-x-\int_0^r b(X_u)\,du
\]
is a continuous local martingale satisfying
\begin{equation}\label{eq:canonical-characteristics}
 \langle M,W\rangle_r=\int_0^r\alpha_u\,du,
 \qquad
 \langle M\rangle_r=\int_0^r\alpha_u\alpha_u^*\,du.
\end{equation}
These identities imply $M=\int\alpha\,dW$: the difference of the two
continuous local martingales has zero quadratic variation. Thus
\begin{equation}\label{eq:canonical-value}
 T_t\varphi(x)=\sup_{P\in\mathfrak P(x)}E^P[\varphi(X_t)].
\end{equation}

For a path after time $s$, use the shifted coordinates
\[
 X^{s}_r=X_{s+r},\qquad
 W^{s}_r=W_{s+r}-W_s,\qquad
 \alpha^{s}_r=\alpha_{s+r},\qquad r\geq0.
\]

\begin{lemma}
\label{lem:DPP-stability}
The canonical family $(\mathfrak P(x))_{x\in\R^d}$ has the following
properties.
\begin{enumerate}[label=\textup{(\roman*)},leftmargin=2.1em]
\item If $P\in\mathfrak P(x)$ and $P^{s,\omega}$ is a regular conditional
law of $(X^s,W^s,\alpha^s)$ given the canonical $\sigma$-field at time $s$,
then
\begin{equation}\label{eq:conditioning-stability}
 P^{s,\omega}\in\mathfrak P(X_s(\omega))
 \quad\text{for $P$-almost every $\omega$}.
\end{equation}

\item Let $K:z\mapsto K_z$ be a universally measurable probability kernel
such that $K_z\in\mathfrak P(z)$ for every $z$.  For
$P\in\mathfrak P(x)$ one can keep the law $P$ up to time $s$ and, conditionally
on the past, append the law $K_{X_s}$.  The resulting concatenated law,
denoted $P\otimes_sK$, belongs to $\mathfrak P(x)$ and satisfies
\begin{equation}\label{eq:pasting-payoff}
 E^{P\otimes_sK}[\varphi(X_{s+t})]
 =E^P\!\left[E^{K_{X_s}}[\varphi(X_t)]\right]
\end{equation}
for every bounded Borel $\varphi$.

\item For every $t>0$, $\varphi\in B_b(\R^d)$ and $\varepsilon>0$, there
is a universally measurable kernel $K^\varepsilon$ with
$K_z^\varepsilon\in\mathfrak P(z)$ and
\begin{equation}\label{eq:epsilon-selector}
 E^{K_z^\varepsilon}[\varphi(X_t)]
 \geq T_t\varphi(z)-\varepsilon,
 \qquad z\in\R^d.
\end{equation}
\end{enumerate}
\end{lemma}

\begin{proof}
We include the measurable-law argument.  Since $\Theta$ is Borel in a
finite-dimensional Euclidean space, it is a standard Borel space; after
choosing a compatible Polish realization, $\mathbb U$ and $\Omega^c$ are
standard Borel.  Hence regular conditional probabilities exist.

The graph
\begin{equation}\label{eq:admissible-graph}
 \mathcal G=\{(z,Q):z\in\R^d,\ Q\in\mathfrak P(z)\}
\end{equation}
is analytic.  Here is a direct verification of the point that matters.
Brownian motion and the two local-martingale identities
\eqref{eq:canonical-characteristics} can be tested after localization at
\[
 \tau_n=\inf\{r:|X_r|+|W_r|\geq n\}\wedge n
\]
at rational pairs of times, against bounded cylinder functions from a
countable determining algebra.  Each resulting expectation is a Borel
function of $(z,Q)$.  Requiring the corresponding equalities for all
$n$ and all tests gives a Borel martingale-problem description of the
admissible laws (and therefore, in particular, an analytic graph).  The
restriction $\alpha_r\in\Theta$ for Lebesgue-almost every $r$ is Borel on
$\mathbb U$.  Conversely, the martingale characterization and
\eqref{eq:canonical-characteristics} reconstruct the weak admissible system,
as observed before the lemma.  This proves the claimed graph property.

For (i), condition each of the countably many stopped martingale identities
at time $s$ and apply the tower property.  After subtracting $W_s$ and
restarting the drift integral at $X_s(\omega)$, the same identities hold
under $P^{s,\omega}$ for almost every $\omega$.  In particular, $W^s$ is
Brownian for the shifted canonical filtration, and
\eqref{eq:canonical-characteristics} holds with initial state
$X_s(\omega)$.  This is precisely \eqref{eq:conditioning-stability}.

For (ii), concatenate paths by keeping $(X,W,\alpha)$ through time $s$ and
then setting
\[
 W_{s+r}=W_s+W'_r,\qquad X_{s+r}=X'_r,
 \qquad\alpha_{s+r}=\alpha'_r, r>0,
\]
where, conditionally on the past, $(X',W',\alpha')$ has law $K_{X_s}$.
The Ionescu--Tulcea construction applies on the standard Borel canonical
space (universal measurability is understood after completion under $P$).
For every one of the stopped martingale tests used above, split the increment
at $s$.  The part before $s$ has mean zero under $P$; conditionally on the
past, the part after $s$ has mean zero under $K_{X_s}$.  The tower property
therefore gives the same martingale and bracket identities under the pasted
law.  Hence $P\otimes_sK\in\mathfrak P(x)$.  Applying the defining conditional
law to the terminal coordinate gives \eqref{eq:pasting-payoff}.

Finally, for fixed $t$ and bounded Borel $\varphi$, the map
\[
 (z,Q)\longmapsto E^Q[\varphi(X_t)]
\]
is Borel on $\mathcal G$.  The measurable maximum theorem for an analytic
graph (equivalently, the Jankov--von Neumann selection theorem applied to
the strict $\varepsilon$-level sets) yields a universally measurable
$\varepsilon$-maximizing selector $z\mapsto K_z^\varepsilon$.  Replacing
$\varepsilon$ by $\varepsilon/2$ in the strict selection gives the
nonstrict inequality \eqref{eq:epsilon-selector}.  This proves (iii).
\end{proof}

\begin{theorem}
\label{thm:DPP}
Under (H1) and (H2), the operators (\ref{nuty}) satisfy, on $B_b(\R^d)$,
\begin{equation}\label{eq:DPP}
 T_{s+t}\varphi=T_s(T_t\varphi),
 \qquad s,t\geq0.
\end{equation}
In particular, $(T_t)_{t\geq0}$ is a nonlinear Markov semigroup and
$T_0$ is the identity.
\end{theorem}

\begin{proof}
The cases $s=0$ or $t=0$ are immediate.  Let $s,t>0$ and put
$g=T_t\varphi$.  Proposition~\ref{prop:regularity} shows that $g$ is bounded
and Lipschitz, hence Borel, so $T_sg$ is defined by
\eqref{eq:canonical-value}.

Fix $P\in\mathfrak P(x)$.  By Lemma~\ref{lem:DPP-stability}(i), the
conditional continuation law at time $s$ is an admissible law starting from
$X_s$.  Therefore
\begin{equation*}
 E^P[\varphi(X_{s+t})]
 =E^P\!\left[
     E^{P^{s,\omega}}[\varphi(X_t)]
   \right]
 \leq E^P[g(X_s)]
 \leq T_sg(x).   
\end{equation*}
Taking the supremum over $P\in\mathfrak P(x)$ proves
\begin{equation}\label{eq:DPP-upper}
 T_{s+t}\varphi(x)\leq T_s(T_t\varphi)(x).
\end{equation}

For the reverse inequality, fix $\varepsilon>0$ and choose
$P^\varepsilon\in\mathfrak P(x)$ such that
\begin{equation*}\label{eq:first-stage-epsilon}
 E^{P^\varepsilon}[g(X_s)]
 \geq T_sg(x)-\varepsilon.
\end{equation*}
Lemma~\ref{lem:DPP-stability}(iii) gives a kernel $K^\varepsilon$ satisfying
\eqref{eq:epsilon-selector}.  Paste it to $P^\varepsilon$ at time $s$.
By parts (ii) and (iii) of that lemma,
\begin{align*}
 T_{s+t}\varphi(x)
 &\geq E^{P^\varepsilon\otimes_sK^\varepsilon}
             [\varphi(X_{s+t})]\\
 &=E^{P^\varepsilon}\!\left[
       E^{K_{X_s}^\varepsilon}[\varphi(X_t)]
    \right]\\
 &\geq E^{P^\varepsilon}[g(X_s)]-\varepsilon\\
 &\geq T_sg(x)-2\varepsilon.
\end{align*}
Letting $\varepsilon\downarrow0$ proves the reverse of
\eqref{eq:DPP-upper}.  No optimizer is used at either stage.
\end{proof}

Now we are ready to prove the following key positive-time joint continuity result.

\begin{corollary}
\label{cor:joint-continuity}
For $\varphi\in B_b(\R^d)$, $(t,x)\mapsto T_t\varphi(x)$ is continuous on
$(0,\infty)\times\R^d$.  On compact subsets of $x$ and time intervals
bounded away from zero it is locally $1/2$-H\"older in time.
\end{corollary}

\begin{proof}
Fix $t>0$ and set $g=T_t\varphi$, which is Lipschitz by
Proposition~\ref{prop:regularity}.  By \eqref{eq:DPP},
$T_{t+h}\varphi=T_hg$, and hence
\[
 |T_{t+h}\varphi(x)-T_t\varphi(x)|
 \leq[g]_{\Lip}
 \sup_{\theta\in\mathcal A^\Theta}
 E|X_h^{x,\theta}-x|.
\]
Condition (H2) gives linear growth of $b$, while
(H1) bounds the diffusion energy.  The standard moment estimate
therefore yields, locally uniformly in $x$,
\[
 \sup_{\theta\in\mathcal A^\Theta}
 E|X_h^{x,\theta}-x|
 \leq C(1+|x|)\sqrt h,
 \qquad 0<h\leq1.
\]
Together with \eqref{eq:Bb-Lip}, this proves the assertion.
\end{proof}
}

{
Now we come back to consider the backward semigroup in the following theorem only. For any fixed $0\leq T <
\infty$, $0 \leq t <T$, any $\varphi \in B_b(\mathbb{R}^d)$, we define a new function $\nu(t,x)$,
\begin{equation}\label{nuty}
    \nu(t,x) = (T_{T,t}\varphi)(x)\coloneqq \underset{\theta \in  \mathcal{A}^{\Theta}_{t,T}}{{\rm sup}} E_P[\varphi(x+\int^T_t b(X_s^{t,x,\theta})ds+\int_t^T\theta_sdW_s],
\end{equation}
where $X_s^{t,x,\theta}=x+\int^s_t b(X_r^{t,x,\theta})dr+\int_t^s\theta_rdW_r$ for any $t\leq s\leq T$ and $T_{T,t}$ are sublinear expectation operators.}

\begin{theorem}\label{revth}
For any fixed $0\leq T<\infty$, $0\leq t < T$, let 
\begin{equation}\label{yTt}
    X_T^{t,x}=x+\int_t^Tb(X_s^{t,x})ds+B_T-B_t,
\end{equation}
where $X_t^{t,x}=x $ and $x\in \mathbb{R}^d$. Then under assumptions \textbf{(H1)} and \textbf{(H2)}, for any $\varphi \in B_b(\mathbb{R}^d)$, we have,
 \begin{equation*}
  \bar{\nu}(t,x)=\bar{E}[\varphi(X_T^{t,x})]{=\nu(t,x)}=\underset{\theta \in  \mathcal{A}^{\Theta}_{t,T}}{{\rm sup}} E_P[\varphi(x+\int^T_t b(X_s^{t,x,\theta})ds+\int_t^T\theta_sdW_s)],
 \end{equation*}
is the viscosity solution of (\ref{mianpdeinap}).
\end{theorem}
\begin{proof}
    { From Corollary \ref{cor:joint-continuity}, in the conext of the backward semigroup,
    } we know that $\nu(t,x)$ is continuous in both time and space for any fixed $0 \leq T< \infty$, $0 \leq t<T  $, $x \in \mathbb{R}^d$, $\varphi \in B_b(\mathbb{R}^d)$.
    
    Let $({T}_{T,t}\varphi)(x)= \nu(t,x)= \underset{\theta \in  \mathcal{A}^{\Theta}_{t,T}}{{\rm sup}} E_P[\varphi(x+\int^T_t b(X_s^{t,x,\theta})ds+\int^T_t\theta_sdW_s)]$.
    Similarly, according to the dynamic programming principle and Yan's commutation theorem, for any $0 \leq t<T  $, there exists a constant $\delta>0$ such that the semigroup property holds: $T_{T,t}=T_{T-\delta,t}T_{T,T-\delta}$.
    It follows from the results in \cite{krylov1986}, \cite{krylov1987} and \cite{pengbook} that $\nu(t,x)$ is a viscosity solution of (\ref{PDE}) or (\ref{mianpdeinap}).
\end{proof}

{ We now trun back to consider $X_t,t\geq 0$, again. Under Assumptions \textbf{(H1)} and \textbf{(H2)}, corresponding result of Theorem {revth} in the forward context says now that
\[\tilde v(t,x)=\bar T_t\varphi(x)=T_t\varphi(x), \varphi\in B_b(\R ^d).\]
This leads us to the proof of the main theorem of this subsection.}

\begin{theorem}\label{xiajie} Suppose the process $X_t$ defined in (\ref{GSDE}) satisfies \textbf{(H1)}, \textbf{(H2)} and \textbf{(H3)}, 
 then the ISE of the SMS $\bar{T_t}$, defined in (\ref{bbbarTv}), exists uniquely.
\end{theorem}
\begin{proof}
   To establish the existence of ISE, it is essential to demonstrate whether the SMS $\bar{T}_t$ defined above satisfies Assumptions \ref{Ass3}, \ref{Ass3.1} and \ref{Ass4}. 
   We set $V(x)=\vert x \vert^2$ and apply $G\text{-}$It\^{o}'s formula to $e^{2\lambda t}V(X_t)$, where $\lambda$ is the constant in \textbf{(H3)}. Then we have,
   \begin{equation*}
      \begin{split}
           & e^{2\lambda T}\vert X_T\vert^2-\vert x \vert^2\\
           =&\int_0^T2\lambda e^{2\lambda s}\vert X_s \vert^2ds + \int_0^T 2e^{2\lambda s}\langle X_s, b(X_s)\rangle ds+ \int_0^T 2e^{2\lambda s} \langle X_ s, dB_s\rangle + \int_0^T e^{2\lambda s}d\langle B \rangle_s\\
           \leq & \int_0^T2\lambda e^{2\lambda s}\vert X_s \vert^2ds + \int_0^T 2e^{2\lambda s}(-\lambda \vert X_s \vert^2+C) ds\\
           &+ \int_0^T 2e^{2\lambda s} \langle X_ s, dB_s\rangle + \int_0^T e^{2\lambda s}d\langle B \rangle_s\\
           =&  \int_0^T 2Ce^{2\lambda s}ds+ \int_0^T 2e^{2\lambda s} \langle X_ s, dB_s\rangle + \int_0^T e^{2\lambda s}d\langle B \rangle_s.\\
      \end{split}
   \end{equation*}
   And $\int_0^T 2e^{2\lambda s} \langle X_ s, dB_s\rangle$ is a $G\text{-}$martingale, so we have $\bar{E}[\int_0^T 2e^{2\lambda s} \langle X_ s, dB_s\rangle]=0$.
By the proposition of the quadratic variation process $\langle B \rangle$, we have 
   \begin{equation*}
           \bar{E}[\int_0^T e^{2\lambda s}d\langle B \rangle_s]\leq\bar{\sigma}^2\bar{E}[\int_0^Te^{2\lambda s}ds ]
    =\frac{\bar{\sigma}^2}{2\lambda}(e^{2\lambda T}-1)  .
   \end{equation*}
So $\bar{E}[e^{2\lambda T}\vert X_T\vert^2]\leq \vert x \vert^2 + (\frac{C}{\lambda}+\frac{\bar{\sigma}^2}{2\lambda})(e^{2\lambda T}-1)$. 
That is $(\bar{T}_TV)(x)=\bar{E}[\vert X_T\vert^2]\leq r\vert x \vert^2 + K $,
where $r=e^{-2\lambda T}$ and $r\in (0,1)$, $K= \frac{2C+\bar{\sigma}^2}{2\lambda}$ and $K>0$.

The above formula shows that $\bar{T}_t$ satisfies Assumptions \ref{Ass3} and \ref{Ass3.1}. Next, we verify Assumption \ref{Ass4}.

{
For every $R,t>0$, define
$\rho_{\sqrt{R}}(t)
 =2\Phi\left(
 \frac{\sqrt{\kappa_0}{\sqrt{R}}}{\sqrt{I_{L_b}(t)}}
 \right)-1.
$
Then $0<\rho_{\sqrt{R}}(t)<1$ and by Lemma \ref{lem:meeting-tail}, the coupling succeeds by time $t$ with probability at least
$\varepsilon_{\sqrt{R}}(t)=1-\rho_{\sqrt{R}}(t)>0$, uniformly over
$x,y\in B_{\sqrt{R}}$ and all admissible controls.
This is because that for $x,y\in B_{\sqrt{R}}$, $|x-y|\leq2{\sqrt{R}}$.  

Fix a skeleton time $t_*>0$ and write $T=T_{t_*}$,
$a=a_{t_*}$ and $K=K_{t_*}$.  Choose
$
 R>\frac{2K}{1-a}.
$
On the pair sublevel set
$
 \mathcal C_R=\{(x,y):V(x)+V(y)\leq R\},
$
we have $|x-y|\leq2\sqrt R$.  Coupling
therefore gives a number $\varepsilon\in(0,1)$ such that
\begin{equation*}\label{eq:small-pair}
 \sup_{(x,y)\in\mathcal C_R}
 \sup_{\theta\in\mathcal A^\Theta}
 \PP(\tau>t_*)\leq1-\varepsilon.
\end{equation*}
The rest of the theorem follows from the same arguments as in Theorem \ref{S1contraction} applying Theorem \ref{thm:direct-coupling}
}
\end{proof}

\end{subsection}

\section{The proof of Lemma \ref{mianlemma}}

We will give the proof of Lemma \ref{mianlemma}, which is based on Section 3 in Denis et al. \cite{DenishU}. We use a similar framework to extend their result by the nonlinear Feynman-Kac formula. For convenience, we use the Einstein notation.
First, we will give some results on the nonlinear Feynman-Kac formula in Chapter 5, Peng \cite{pengbook}.

Let $\Omega=C_0^d(\mathbb{R}^+)$ be the space of all $\mathbb{R}^d$-valued continuous paths $(\omega_t)_{t\geq 0}$, with $\omega_0=0$, $\Omega_T=\{\omega_{.\wedge T}:\omega\in\Omega \}$. $(B_s)_{s\geq0 }$ is the d-dimensional $G\text{-}$Brownian motion.

Consider the following $G\text{-}$SDE, for fixed $0 \leq t \leq T < \infty$, $s \in [t,T]$,
\begin{equation}
       d X^{t,\xi}_s = b(X^{t,\xi}_s)ds+h_{ij}(X^{t,\xi}_s)d \left \langle B \right \rangle^{ij}_s+ \sigma _j(X^{t,\xi}_s)dB^j_s,\ \ 
       X^{t,\xi}_t=\xi,\\
\end{equation}
where $\xi \in L^2_G(\Omega_t; \mathbb{R}^d)$ and $b$, $h_{ij}$, $\sigma_j$: $\mathbb{R}^d \mapsto \mathbb{R}^d$ are given Lipschitz functions.

Then consider the associated $G\text{-}$BSDE,
\begin{equation}
    Y_s^{t,\xi}=\bar{E}[\Phi(X^{t,\xi}_T)+\int_s^T f(X^{t,\xi}_r,Y_r^{t,\xi})dr + \int_s^T g_{ij}(X^{t,\xi}_r,Y_r^{t,\xi})d \left \langle B^i,B^j \right \rangle_r |\Omega_s],
\end{equation}
where $\Phi : \mathbb{R}^d \to \mathbb{R}$, $f$, $g_{ij}: \mathbb{R}^d \times \mathbb{R} \mapsto \mathbb{R} $ are given Lipschitz functions.
\begin{proposition}(\cite{pengbook})
    For each $\xi,\xi^{\prime} \in L_G^2(\Omega_t;\mathbb{R}^d)$, we have, for each $s \in [t,T]$,
    \begin{equation}
        \bar{E}[\vert X^{t,\xi}_s - X^{t,\xi^{\prime}}_s\vert ^2 | \Omega_t]\leq C\vert \xi-\xi^{\prime}\vert^2
    \end{equation}
and 
\begin{equation}
    \bar{E}[\vert X^{t,\xi}_s\vert ^2 | \Omega_t]\leq C(1+ \vert \xi\vert^2),
\end{equation}
where the constant $C$ depends only on the Lipschitz constant.
\end{proposition}

\begin{corollary}(\cite{pengbook})\label{coA.2}
    For any $\xi \in L_G^2(\Omega_t;\mathbb{R}^d)$, we have 
    \begin{equation}
        \bar{E}[\vert X^{t,\xi}_{t+\delta}-\xi\vert ^2 | \Omega_t] \leq C(1+\vert \xi \vert^2)\delta\ for \ \delta\in[0,T-t],
    \end{equation}
    where the constant $C$ depends only on the Lipschitz constants of $b$, $h$ and $\sigma$. 
\end{corollary}

 When $\xi = x \in \mathbb{R}^d$, define
 \begin{equation}
     u(t,x)\coloneqq Y^{t,x}_t, \ (t,x)\in[0,T]\times\mathbb{R}^d.
 \end{equation}
 For any $A\in \mathbb{S}(d)$, $p \in \mathbb{R}^d$, $r \in \mathbb{R}$, denote
\begin{equation*}
    F(A,p,r,x)\coloneqq G(B(A,p,r,x))+\left\langle p,b(x)\right\rangle+f(x,r),
\end{equation*}
where for any $Q \in \mathbb{S}(d) $, $G(Q)=\frac{1}{2}\bar{E}[\left\langle QB_1,B_1\right\rangle]$, and $B(A,p,r,x)$ is a $d\times d$ symmetric matrix with
\begin{equation*}
    B_{ij}(A,p,r,x)\coloneqq \left\langle A\sigma_i(x), \sigma_j(x)\right\rangle+ \left\langle p,h_{ij}(x)+h_{ji}(x)\right\rangle+g_{ij}(x,r)+g_{ji}(x,r).
\end{equation*}
\begin{theorem}(\cite{pengbook})
  The function $u(t,x)$ is the unique viscosity solution of the following PDE,
  \begin{equation}
          \partial_tu+F(D^2u,Du,u,x)=0,\ \ 
          u(T,x)=\Phi(x).\\
  \end{equation}
  where $D\coloneqq (\partial_{x_i})^d_{i=1}$, $D^2 \coloneqq (\partial_{x_ix_j}^2)^d_{i,j=1}$.
\end{theorem}

Now we begin the proof of Lemma 4.5. Here we only consider the $G\text{-}$SDE such that $h_{ij}=0$ for $i,j = 1,..d$ and $\sigma=\sigma_{i,j}^d$ is a unit matrix.
 \begin{equation}
        d X^{t,\xi}_s = b(X^{t,\xi}_s)ds+dB_s,\ \
        X^{t,\xi}_t=\xi,\\
\end{equation}
where $\xi \in L^2(\Omega_t; \mathbb{R}^d)$ and $b: \mathbb{R}^d \mapsto \mathbb{R}^d$ are given Lipschitz functions. Let $\Theta$ be a given bounded and closed subset in $\mathbb{R}^{d \times d}$, $\mathcal{A}^{\Theta}_{t,T}$ to be the collection of all $\Theta$-valued, $\mathbb{F}$-adapted processes on an interval $[t,T]$, $0 \leq t \leq T < \infty$ and for $\theta \in \mathcal{A}^{\Theta}_{t,T}$, consider
\begin{equation}
        d X^{t,\xi,\theta}_s = b(X^{t,\xi, \theta}_s)ds+\theta_sdW_s,\ \
       X^{t,\xi, \theta}_t=\xi,
\end{equation}
where $(W_s)_{s\geq 0}$ is a d-dimentional standard Brownian motion. We set $\mathcal{F}_s$ as the natural filtration generated by $W$ and $\mathbb{F}=\{\mathcal{F}_s\}_{s>0}$.

Given $\varphi \in C_{b,lip}(\mathbb{R}^d)$, 
let
\begin{equation}
    \Lambda_{t,T}[\xi]=\underset{\theta \in \mathcal{A}^{\Theta}_{t,T}}{\rm esssup}E_P[\varphi(X^{t,\xi, \theta}_T)|\mathcal{F}_t].
\end{equation}

\begin{lemma}\label{LeA.4}
    For each $\theta^1$ and $\theta^2$ in $\mathcal{A}^{\Theta}_{t,T}$, there exists $\theta \in \mathcal{A}^{\Theta}_{t,T}$ such that
\begin{equation}
      E_P[\varphi(X^{t,\xi, \theta}_T)|\mathcal{F}_t] =E_P[\varphi(X^{t,\xi, \theta^1}_T)|\mathcal{F}_t]\vee E_P[\varphi(X^{t,\xi, \theta^2}_T)|\mathcal{F}_t]
    \end{equation}
Consequently, there exists a sequence $\{ \theta^i\}^\infty_{i=1} $ of $\mathcal{A}^{\Theta}_{t,T}$ such that
\begin{equation}
    E_P[\varphi(X^{t,\xi, \theta^i}_T)|\mathcal{F}_t]\nearrow  \Lambda_{t,T}[\xi],\ P\text{-}a.s.
\end{equation}
We also have, for each $r \leq t$,
\begin{equation}
  E_P[\underset{\theta \in \mathcal{A}^{\Theta}_{t,T}}{\rm esssup}E_P[\varphi(X^{t,\xi, \theta}_T)|\mathcal{F}_t]|\mathcal{F}_r]=\underset{\theta \in \mathcal{A}^{\Theta}_{t,T}}{\rm esssup}E_P[\varphi(X^{t,\xi, \theta}_T)|\mathcal{F}_r]  
\end{equation}
\end{lemma}

We refer to a similar proof in \cite{DenishU} (Lemma 40) and Yan's commutation theorem \cite{Yanjiaan}, so we omit it.
  
The next lemma gives some properties about the function $\Lambda_{t,T}[\cdot]$. A similar proof can be found in \cite{DenishU}.
\begin{lemma}\label{liplemma}
    The function $\Lambda_{t,T}[\cdot]: L^2(\Omega,\mathcal{F}_t,P;\mathbb{R}^d) \to L^2(\Omega,\mathcal{F}_t,P;\mathbb{R})$ has the following properties, for each $\xi$, $\xi^{\prime}\in L^2(\mathcal{F}_t)$,\\
\rm{(1)} $\Lambda_{t,T}[\xi]\leq C_{\varphi}$;\\
\rm{(2)} $\vert \Lambda_{t,T}[\xi]-\Lambda_{t,T}[\xi^{\prime}]\vert\leq K_{\varphi}\vert \xi-\xi^{\prime} \vert$,\\
where $C_{\varphi} = \underset{x}{{\rm sup}} \ \varphi(x)$ and $K_\varphi$ is the Lipschitz constant of $\varphi$. 
\end{lemma}

\begin{lemma}
    For each $x \in \mathbb{R}^d$, $\Lambda_{t,T}[x]$ is a deterministic function.
\end{lemma}

\begin{proof}
We set for any fixed $t\geq0$, $\mathcal{F}^t_r\coloneqq \sigma\{W_{t+u}-W_t,\ 0\leq u\leq r\}\vee \mathcal{N}$, $\mathbb{F}^t\coloneqq \sigma(\cup_{r=0}^{\infty}\{\mathcal{F}_r^t\})$.
where $\mathcal{N}$ is the collection of $P$-null subsets. We can see that $\mathbb{ F}^t$ is independent of $\mathcal{F}_t$. Let
\begin{equation*}
\begin{split}
    A_{\theta}\coloneqq\{& \theta_s= \Sigma_{j=1}^N \mathds{1}_{A^j}\theta_s^j:\{A^j\}_{j=1}^N \ is\  an\  \mathcal{F}_t\text{-}partition \ of\  \Omega,\\
    &\theta^j \in \mathcal{A}^{\Theta}_{t,T}\ is \ 
(\mathbb{F}^t)\ \text{-} adapted \}.\\
\end{split}
\end{equation*}
Then the collection of processes $(\theta_s)_{s\in[t,T]}$, $A_{\theta}$ is dense in $\mathcal{A}_{t,T}^{\Theta}$. According to Lemma \ref{LeA.4}, we can take $\theta_{n,s} = \Sigma^{N_n}_{j=1}\mathds{1}_{A_{n}^j}\theta_{n,s}^{j}$ such that 
$E_P[\varphi(X_T^{t,x,\theta_n})|\mathcal{F}_t] \nearrow \Lambda_{t,T}[x]$. 

 Consider $X_T^{t,x,\theta_n^j}=x +\int_t^T b(X_r^{t,x_j,\theta^{j}_n})dr+\int_t^T\theta^{j}_{n,r}dW_r$
and we need to prove that
\begin{equation}\label{In}
  \Sigma^{N_n}_{j=1}\mathds{1}_{A^{j}_n}X_T^{t,x,\theta_n^j}=X_T^{t,x,\theta_n}.
\end{equation}
For this, we see that, 
\begin{equation*}
   \begin{split}
       & \Sigma^{N_n}_{j=1}\mathds{1}_{A^{j}_n}X_T^{t,x,\theta_n^j} = \Sigma^{N_n}_{j=1}\mathds{1}_{A^{j}_n}(x +\int_t^T b(X_r^{t,x_j,\theta^{j}_n})dr+\int_t^T\theta^{j}_{n,r}dW_r)\\
       =& x + \int_t^T b(\Sigma^{N_n}_{j=1}\mathds{1}_{A^{j}_n}X_r^{t,x,\theta_n})dr+\int_t^T \theta_{n,r}dW_r.
   \end{split}
\end{equation*}
Recall $X_T^{t,x,\theta_n}=x+\int_t^Tb(X_r^{t,x,\theta_n})dr+\int_t^T\theta_{n,r}dW_r$. By the existence and uniqueness of the solution of SDE, we can prove (\ref{In}).

It then follows that

\begin{equation*}
    \begin{split}
        &E_P[\varphi(X_T^{t,x,\theta_n})|\mathcal{F}_t]
        =E_P[\varphi(x+\int_t^Tb(X_r^{t,x,\theta_n})dr+\int_t^T\theta_{n,r}dW_r)|\mathcal{F}_t]\\
        =&\Sigma^{N_n}_{j=1}\mathds{1}_{A^{j}_n}E_P[\varphi(x +\int_t^T b(X_r^{t,x,\theta^{j}_n})dr+\int_t^T\theta^{j}_{n,r}dW_r)|\mathcal{F}_t]\\
        =&\Sigma^{N_n}_{j=1}\mathds{1}_{A^{j}_n}E_P[\varphi(x +\int_t^T b(X_r^{t,x,\theta^{j}_n})dr+\int_t^T\theta^{j}_{n,r}dW_r)]\\
        \leq &\underset{1\leq j \leq N_i}{\rm max}E_P[\varphi(x +\int_t^T b(X_r^{t,x,\theta^{j}_n})dr+\int_t^T\theta^{j}_{n,r}dW_r)].
    \end{split}
\end{equation*}

We set $j_n$ as the maximizer of $\{ E_P[\varphi(X_t^{t,x,\theta^{j}_n})]\}_{j=1}^{N_n}$, i.e. we have 
\begin{equation*}
  E_P[\varphi(X_T^{t,x,\theta_n})|\mathcal{F}_t]\leq  E_P[\varphi(X_t^{t,x,\theta^{j_n}_n})] .
\end{equation*}
Taking $n \to \infty$ on both sides, we have $ \Lambda_{t,T}[x] \leq  \underset{n\to\infty}{lim}E_P[\varphi(X_t^{t,x,\theta^{j_n}_n})]$.
From the definition of $ \Lambda_{t,T}[x] $, we have $\Lambda_{t,T}[x] \geq  \underset{n\to\infty}{lim}E_P[\varphi(X_t^{t,x,\theta^{j_n}_n})]$.

So, $ \underset{n\to\infty}{lim}E_P[\varphi(X_t^{t,x,\theta^{j_n}_n})]=\Lambda_{t,T}[x],\ a.s.$
and thus $\Lambda_{t,T}[x]$ is deterministic. 
\end{proof}

Denote $\mu_{t,T}(x)\coloneqq \Lambda_{t,T}[x]$, $t \leq T$. Then $\mu_{t,T}(\cdot)$ is a bounded Lipschitz function.

\begin{lemma}\label{LeA.7}
    For each $\xi \in L^2_G(\Omega,\mathcal{F}_t,P;\mathbb{R}^d)$, we have $ \mu_{t,T}(\xi)=\Lambda_{t,T}[\xi],\ a.s.$.
\end{lemma}
\begin{proof}
    Firstly, we prove the situation that $\xi$ is a step function, $\xi=\Sigma_{j=1}^N\mathds{1}_{A_j}x_j$, $x_j \in \mathbb{R}^d$, $\{A_j\}_{j=1}^N$ is a $\mathcal{F}_t$-partition of $\Omega$. By the proof of Lemma A.6, for each $x_j$, we can find $\{\theta^{j_n}_n \}_{n=1}^{\infty}$ of $\mathcal{A}^{\Theta}_{t,T}$ be $\mathcal{F}_s^t \text{-}$adapted process such that
    \begin{equation*}
        \underset{n\to \infty}{lim}E_P[\varphi(X_T^{t,x_j,\theta^{j_n}_n})]
         =\Lambda_{t,T}[x_j] =\mu_{t,T}(x_j).
    \end{equation*}
Setting $\theta_n=\Sigma_{j=1}^N\mathds{1}_{{A}_j}\theta^{j_n}_n$, we have
\begin{equation*}
   \begin{split}
        \Lambda_{t,T}[\xi]& \geq E_P[\varphi(X_T^{t,\xi,\theta_n})|\mathcal{F}_t]\\
    &=\Sigma^{N}_{j=1}\mathds{1}_{A_{j}}E_P[\varphi(x_j +\int_t^T b(X_r^{t,x_j,\theta^{j_n}_n})dr+\int_t^T\theta^{j_n}_{n,r}dW_r)|\mathcal{F}_t]\\
    &=\Sigma^{N}_{j=1}\mathds{1}_{A_{j}}E_P[\varphi(x_j +\int_t^T b(X_r^{t,x_j,\theta^{j_n}_n})dr+\int_t^T\theta^{j_n}_{n,r}dW_r)]\\
    &\to \Sigma^{N}_{j=1}\mathds{1}_{A_{j}}\mu_{t,T}(x_j)=\mu_{t,T}(\xi).\\
   \end{split}
\end{equation*}
Moreover, we also have for each $\theta \in \mathcal{A}_{t,T}^\Theta$,
\begin{equation*}
    \begin{split}
        E_P[\varphi(X_T^{t,\xi,\theta^i})|\mathcal{F}_t]
        =&\Sigma^{N}_{j=1}\mathds{1}_{A_{j}}E_P[\varphi(x_j +\int_t^T b(X_r^{t,x,\theta^{j_n}_n})dr+\int_t^T\theta^{j_n}_{n,r}dW_r)|\mathcal{F}_t]\\
        \leq& \Sigma^{N}_{j=1}\mathds{1}_{A_j}\mu_{t,T}(x_j)=\mu_{t,T}(\xi).\\ 
    \end{split}
\end{equation*}
    So we have $\Lambda_{t,T}[\xi] = \mu_{t,T}(\xi)$.

For the case of a general $\xi \in L_G^2(\Omega,\mathcal{F}_t,P;\mathbb{R}^d)$ and $\xi \geq 0$, we can find a sequence of step functions $\{\xi_n\}_{n\geq1}$ such that $\xi_n \uparrow \xi$, as $n \to \infty$. Then
\begin{equation*}
    \begin{split}
        \vert \Lambda_{t,T}[\xi]-\mu_{t,T}(\xi)\vert
        =&\vert \Lambda_{t,T}[\xi]-\Lambda_{t,T}[\xi_n]+\Lambda_{t,T}[\xi_n]-\mu_{t,T}(\xi_n)+\mu_{t,T}(\xi_n)-\mu_{t,T}(\xi)\vert\\
        \leq & 2K_{\varphi}\vert \xi-\xi_n \vert \to 0,\\
    \end{split}
\end{equation*}
as $n\to \infty$, a.s., where $K_{\varphi}$ is the Lipschitz constant of $\varphi$.
\end{proof}

 Now we give the dynamic programming principle.
For any $0 \leq t\leq r\leq T<\infty$, let $ \hat{X}_{r,T}^{t,\xi,\theta}=X_{T}^{t,\xi,\theta}-X_{r}^{t,\xi,\theta}$.
It is clear that
\begin{equation*}
    \underset{\theta\in\mathcal{A}_{t,T}^{\Theta}}{ {\rm esssup}}E_P[\varphi(X_T^{t,\xi,\theta})|\mathcal{F}_t]=\underset{\theta\in\mathcal{A}_{t,r}^{\Theta}}{{\rm esssup}}\left\{ \underset{\bar{\theta}\in\mathcal{A}_{r,T}^{\Theta}}{{\rm esssup}}E_P[\varphi(X_r^{t,\xi,\theta}+\hat{X}_{r,T}^{t,\xi,\bar{\theta}})|\mathcal{F}_t] \right\}.
\end{equation*}
According to Lemma \ref{LeA.4} and Lemma \ref{LeA.7}, we have
\begin{equation*}
    \underset{\bar{\theta}\in\mathcal{A}_{r,T}^{\Theta}}{{\rm esssup}}E_P[\varphi(X_r^{t,\xi,\theta}+\hat{X}_{r,T}^{t,\xi,\bar{\theta}})|\mathcal{F}_t]=E_P[(\underset{\bar{\theta}\in\mathcal{A}_{r,T}^{\Theta}}{{\rm esssup}}E_P[\varphi(X_r^{t,\xi,\theta}+\hat{X}_{r,T}^{t,\xi,\bar{\theta}})|\mathcal{F}_r] |\mathcal{F}_t)].
\end{equation*}
So 
\begin{equation*}
  \underset{\theta\in\mathcal{A}_{t,T}^{\Theta}}{{\rm esssup}}E_P[\varphi(X_T^{t,\xi,\theta})|\mathcal{F}_t]=  \underset{\theta\in\mathcal{A}_{t,r}^{\Theta}}{{\rm esssup}}E_P[(\underset{\bar{\theta}\in\mathcal{A}_{r,T}^{\Theta}}{{\rm esssup}}E_P[\varphi(X_r^{t,\xi,\theta}+\hat{X}_{r,T}^{t,\xi,\bar{\theta}})|\mathcal{F}_r] |\mathcal{F}_t)].
\end{equation*}
For each $\varphi \in C_{b,lip}(\mathbb{R}^d)$ and $(t,x) \in [0,T]\times\mathbb{R}^d$, we set $\nu(t,x)\coloneqq \underset{\theta\in\mathcal{A}^{\Theta}_{t,T}}{{\rm sup}}E_P[\varphi(X_T^{t,x,\theta})].$
And for each $h \in [0,T-t]$, according to the equality above,
\begin{equation*}
    \begin{split}
        \nu(t,x)&=\underset{\theta\in\mathcal{A}^{\Theta}_{t,T}}{{\rm sup}}E_P[\varphi(X_T^{t,x,\theta})]\\
        &=\underset{\theta\in\mathcal{A}^{\Theta}_{t,t+h}}{{\rm sup}}E_P[\underset{\theta\in\mathcal{A}^{\bar{\theta}}_{t+h,T}}{{\rm sup}}E_P[\varphi(X_{t+h}^{t,x,\theta}+\int_{t+h}^T b(X_s^{t+h,X_{t+h}^{t,x,\theta},\bar{\theta}})ds+\int_{t+h}^T\bar{\theta}dW_s)]]\\
        &=\underset{\theta\in\mathcal{A}^{\Theta}_{t,t+h}}{{\rm sup}}E_P[\nu(t+h,X_{t+h}^{t,x,\theta} )].
    \end{split}
\end{equation*}
\begin{proposition}\label{ProA.8}
For any $h \in [0,T-t]$,
  \begin{equation}\label{Anu}
     \nu(t,x)= \underset{\theta\in\mathcal{A}^{\Theta}_{t,t+h}}{{\rm sup}}E_P[\nu(t+h,X_{t+h}^{t,x,\theta} )].
  \end{equation}  
\end{proposition}
\begin{lemma}
    The function $\nu$ defined by (\ref{Anu}) is bounded by $C_{\varphi}$ and is a Lipschitz function in $x$ and $\frac{1}{2}\text{-}h\ddot{\rm o}lder$ function in $t$.
\end{lemma}
\begin{proof}
  It is easy to see that $\nu$ is bounded by $C_{\varphi}$ and $\nu$ is a Lipschitz function in $x$ according to Lemma \ref{liplemma} since $\varphi$ is Lipchitz. So we only need to prove $\nu$ is $\frac{1}{2}\text{-}h\ddot{\rm o}lder$ function in $t$. For any $h \in [0,T-t]$,
  \begin{equation*}
      \begin{split}
        \vert  \nu(t,x)-\nu(t+h,x) \vert
        \leq &\underset{\theta\in\mathcal{A}^{\Theta}_{t,t+h}}{{\rm sup}}E_P[\vert \nu(t+h,X_{t+h}^{t,x,\theta} )-\nu(t+h,x)\vert ] \\
        \leq&K_{\varphi}\underset{\theta\in\mathcal{A}^{\Theta}_{t,t+h}}{{\rm sup}}E_P[\vert X_{t+h}^{t,x,\theta}-x \vert].
      \end{split}
  \end{equation*}
According to Corollary \ref{coA.2}, we have
\begin{equation*}
      E_P[\vert X_{t+h}^{t,x,\theta}-x \vert]  
     \leq  (E_P[\vert X_{t+h}^{t,x,\theta}-x \vert^2])^{\frac{1}{2}}
     \leq C(1+\vert x\vert^2)^{\frac{1}{2}}h^{\frac{1}{2}}, 
\end{equation*}
 where the constant $C$ depends only on the Lipschitz constant.
 So $\nu$ is $\frac{1}{2}\text{-}h\ddot{o}lder$ continuous in $t$.
\end{proof}

\begin{theorem}
    The function $\nu$ is a viscosity solution of the following PDE,
    \begin{equation}\label{PDE}
          \partial_t\nu+F(D^2\nu,D\nu,\nu,x)=0,\ \ 
          \nu(T,x)=\varphi(x).
    \end{equation}
    where 
    \begin{equation*}
        F(D^2\nu,D\nu,\nu,x)=G(D^2\nu)+\left\langle D\nu,b(x)\right\rangle.
    \end{equation*}
  
\end{theorem}
\begin{proof}
    For fixed $(t,x) \in [0,T]\times \mathbb{R}^d$, let $\psi \in C_b^{2,3}((0,T)\times\mathbb{R}^d)$ be such that $\psi \geq \nu$ and $\psi(t,x)=\nu(t,x)$. From Proposition \ref{ProA.8} and It\^{o}'s formula, we have that
   \begin{equation}\label{longineq}
       \begin{split}
           0&=\underset{\theta\in\mathcal{A}^{\Theta}_{t,t+h}}{{\rm sup}}E_P[\nu(t+h,X_{t+h}^{t,x,\theta} )-\nu(t,x)]\\
           &\leq \underset{\theta\in\mathcal{A}^{\Theta}_{t,t+h}}{{\rm sup}}E_P[\psi(t+h,X_{t+h}^{t,x,\theta} )-\psi(t,x)]\\
           &=\underset{\theta\in\mathcal{A}^{\Theta}_{t,t+h}}{{\rm sup}}E_P[\int_t^{t+h}\{\frac{\partial\psi}{\partial s}+\left\langle D\psi, b \right\rangle
           +\frac{1}{2}tr(\theta_s\theta^T_sD^2\psi)\}(s,X_s^{t,x,\theta})ds].
       \end{split}
   \end{equation}
Since $(\frac{\partial\psi}{\partial s}+\frac{1}{2}tr[\theta_s\theta^T_sD^2\psi])(s,y)$ is uniformly Lipschitz in $(s,y)$, we have that for small $h > 0$, $t \leq s \leq t+h$,
\begin{equation*}
    \begin{split}
        &E_P[(\frac{\partial\psi}{\partial s}+\frac{1}{2}tr[\theta_s\theta^T_sD^2\psi])(s,X_s^{t,x,\theta})]\\
        \leq & E_P[(\frac{\partial\psi}{\partial s}+\frac{1}{2}tr[\theta_s\theta^T_sD^2\psi])(t,x)+C \vert X_s^{t,x,\theta}-x \vert]+ Ch,
    \end{split}
\end{equation*}
where the constant $C$ depends only on the Lipschitz constant. 

Moreover, $(D\psi)(s,y)$ is also uniformly Lipschitz in $(s,y)$ and bounded, $b(y)$ is Lipschitz in $y$, then,
\begin{equation*}
    \begin{split}
    &\vert \left\langle D\psi(s,X_s^{t,x,\theta}), b(X_s^{t,x,\theta}) \right\rangle   -  \left\langle D\psi(t,x), b(x) \right\rangle    \vert\\
     \leq &  \vert \left\langle D\psi(s,X_s^{t,x,\theta}), b(X_s^{t,x,\theta}) \right\rangle \vert + \vert    \left\langle D\psi(t,x), b(x) \right\rangle    \vert\\
      &\ \ +   \left\langle D\psi(s,X_s^{t,x,\theta}), b(x) \right\rangle-  \left\langle D\psi(t,x), b(x) \right\rangle                     \vert\\
      \leq & \vert \left\langle D\psi(s,X_s^{t,x,\theta}), b(X_s^{t,x,\theta})-b(x) \right\rangle \vert + \vert \left\langle D\psi(s,X_s^{t,x,\theta})-D\psi(t,x), b(x) \right\rangle  \vert\\
      \leq& C\vert D\psi(s,X_s^{t,x,\theta})\vert\vert X_s^{t,x,\theta} - x\vert+C\vert b(x) \vert (\vert X_s^{t,x,\theta} - x\vert+h)\\
      \leq & C \vert X_s^{t,x,\theta}-x \vert]+ Ch,
    \end{split}
\end{equation*}
where $C$ is a constant depending on the Lipschitz constant, the bounds of $\vert D\psi\vert $ and $\vert b(x)\vert$, and may vary from line to line.

From Corollary \ref{coA.2}, we know that $ E_P[\vert X_s^{t,x,\theta}-x \vert] \leq C(1+\vert x\vert^2)^{\frac{1}{2}}h^{\frac{1}{2}}$.
So,it follows from (\ref{longineq}) that 
\begin{equation*}
\begin{split}
   \underset{\theta\in\mathcal{A}^{\Theta}_{t,t+h}}{{\rm sup}}E_P[&\int_t^{t+h}\frac{\partial\psi}{\partial s}(t,x)+\left\langle D\psi(t,x), b(x) \right\rangle
          +\frac{1}{2}tr[\theta_s\theta^T_sD^2\psi](t,x)dx]\\
          &+ Ch^{\frac{3}{2}}+Ch^2\geq 0.\\      
\end{split}
\end{equation*}
Here $C$ is the constant depend on the Lipschitz constant, the bounded of $D\psi$, $b(x)$ and $\vert x \vert$.
Thus, 
\begin{equation*}
 [\frac{\partial\psi}{\partial s}(t,x)+\left\langle D\psi(t,x), b(x) \right\rangle
          +\frac{1}{2}\underset{\gamma \in \Theta}{{\rm sup}}\ tr[\gamma \gamma^TD^2\psi](t,x)]h \geq 0.   
\end{equation*}
It turns out that 
\begin{equation*}
   \partial_t\psi+F(D^2\psi,D\psi,\psi,x)\geq 0.
\end{equation*}
So $\nu$ is a viscosity subsolution of PDE (\ref{PDE}). Similarly, we can prove that it is also a viscosity supersolution, thus it is a viscosity solution.
\end{proof}

By the uniqueness of the viscosity solution (Theorem 2.9 \cite{pengbook}) and the nonlinear Feynman-Kac formula, let $t=0$ and the following Theorem follows easily.
\begin{theorem}
Let $ X_T=x+\int^T_0 b(X_s)ds+B_T$, where $X_0=x \in \mathbb{R}^d$ is the initial value, $b$ is a Lipschitz continuous function, $B$ is a d-dimensional $G\text{-}$Brownian motion. Then for each $\varphi \in C_{b,lip}(\mathbb{R}^{d })$, for fixed $0 \leq T < \infty$, the G-expectation can be equivalently defined by
\begin{equation}
    \bar{E}[\varphi(X_T)]=\underset{\theta \in  \mathcal{A}^{\Theta}_{0,T}}{{\rm sup}} E_P[\varphi(x+\int^T_0 b(X_s^{\theta})ds+\int_0^T\theta_sdW_s],
\end{equation}
where $ X_t^{\theta}=x + \int_0^tb(X_s^{\theta})ds+\int_0^t\theta_sdW_s$.
\end{theorem}

\section*{Acknowledgments}
HZ acknowledges financial support from the EPSRC Established Career Fellowship (EP/S005293/2) and the Newton Fund (NIF \textbackslash R1\textbackslash 221003).


\bibliographystyle{siamplain}
\bibliography{references}
\end{document}